\documentclass[final,leqno]{siamltex}

\usepackage{graphicx}
\usepackage[labelformat=simple]{subcaption}
\usepackage{amsmath,amssymb}
\usepackage{mathtools}
\usepackage{algorithm, algorithmic, multicol}
\usepackage{empheq}
\usepackage{float}

\mathtoolsset{showonlyrefs=true}

\makeatletter
\renewcommand*\env@matrix[1][*\c@MaxMatrixCols c]{
  \hskip -\arraycolsep
  \let\@ifnextchar\new@ifnextchar
  \array{#1}}
\makeatother

\makeatletter
\renewcommand{\ALG@name}{\sc Algorithm}
\makeatother

\newcommand{\ua}{\underline{a}}
\newcommand{\ub}{\underline{b}}
\newcommand{\uc}{\underline{c}}
\newcommand{\ty}{\tilde{y}}
\newcommand{\bigO}{\mathcal{O}}
\newcommand{\md}{\mathrm{d}}
\newcommand{\dt}{\mathrm{d}t}

\newcommand{\dty}{\mathrm{d}\tilde{y}}
\newcommand{\dd}{\mathrm{d}t\, \mathrm{d}(x+1)}

\newcommand{\ixx}{\int_{-1}^{x+1}}
\newcommand{\iy}{\int_{-1}^y}

\newcommand{\iiy}{\int_{-1}^y\hspace{-0.05cm}\int_{-1}^{x+1}\hspace{-0.25cm}}

\let\tilde\widetilde

\title{Spectral Approximation of Convolution Operator}

\author{Kuan Xu\thanks{School of Mathematics, Statistics, and Actuarial Science, 
University of Kent, Canterbury, CT2 7FS, UK (\texttt{k.xu@kent.ac.uk}).  The 
work of this author was supported by the Royal Society under Research Grant 
RG160236.}
\and Ana F. Loureiro\thanks{School of Mathematics, Statistics, and Actuarial 
Science, University of Kent, Canterbury, CT2 7FS, UK 
(\texttt{a.loureiro@kent.ac.uk}).}
}

\begin{document}

\maketitle

\begin{abstract}
We develop a unified framework for constructing matrix approximations to the
convolution operator of Volterra type defined by functions that are approximated 
using classical orthogonal polynomials on $[-1, 1]$. The numerically stable 
algorithms we propose exploit recurrence relations and symmetric properties 
satisfied by the entries of these convolution matrices. Laguerre-based 
convolution matrices that approximate Volterra convolution operator defined 
by functions on $[0, \infty]$ are also discussed for the sake of completeness. 
\end{abstract}

\begin{keywords}
convolution, Volterra convolution integral, operator approximation, orthogonal 
polynomials, Chebyshev polynomials, Legendre polynomials, Gegenbauer 
polynomials, ultraspherical polynomials, Jacobi polynomials, Laguerre 
polynomials, spectral methods
\end{keywords}

\begin{AMS}
44A35, 65R10, 47A58, 33C45, 41A10, 65Q30 
% 44A35   Convolution
% 65R10   Integral transforms
% 47A58   Operator approximation theory
% 33C45   Orthogonal polynomials and functions of hypergeometric type 
%         (Jacobi, Laguerre, Hermite, Askey scheme, etc.)
% 41A10   Approximation by polynomials
% 65Q30   Recurrence relations
\end{AMS}

\thispagestyle{plain}

%%%%%%%%%%%%%%%%%%%%%%%%%%%%%%%%%%%%%%%%%%%%%%%%%%%%%%%%%%%%%%%%%%%%%%%%%%%%%%%%

\section{Introduction}
Convolution operators abound in science and engineering. They can be found in, 
for example, statistics and probability theory \cite{hog}, computer vision 
\cite{for}, image and signal processing \cite{dam}, and system control 
\cite{son}. In applied mathematics, convolution operators figure in many 
topics, from Green's function \cite{duf} to Duhamel's principle \cite{sta}, 
from non-reflecting boundary condition \cite{hag} to large eddy simulation 
\cite{wag}, from approximation theory \cite{tre} to fractional calculus 
\cite{hil}. Furthermore, convolution operators are the key building blocks of 
the convolution integral equations \cite{atk, lin, bru}.  

Given two continuous functions $f(x)$ and $g(x)$ on the interval $[-1,1]$, the 
(left-sided) convolution operator $V$ of Volterra type defined by $f(x)$ is 
given by 
\begin{equation}
V[f](g) = h(x) = \ixx f(x-t) g(t)\dt,~~~~~x \in [-2, 0]. \label{V}
\end{equation}
For more general cases where $f(x):[a, b] \to \mathbb{R}$ and $g(x):[c, d] \to 
\mathbb{R}$ with $b-a = d-c$, the convolution operator can be shown equivalent 
to \eqref{V} via changes of variables and rescaling. Thus we only consider 
\eqref{V} throughout without loss of generality. 

If $f(x)$ and $g(x)$ have a little extra smoothness beyond continuity, they can 
be approximated by polynomials $f_M(x)$ and $g_N(x)$ of sufficiently high degree 
so that $\|f(x) - f_M(x)\|_{\infty}$ and $\|g(x) - g_N(x)\|_{\infty}$ are on 
the order of machine precision \cite{tre, dri}. In this article, we focus on 
the approximation of $V[f]$ when $f(x)$ and $g(x)$ are approximated using the 
orthogonal polynomials of the Jacobi family, e.g. the Chebyshev polynomials:
\begin{equation}
f_M(x) = \sum_{m=0}^M a_m T_m(x) ~~~\mbox{and}~~~ g_N(x) =
\sum_{n=0}^N b_n T_n(x), \label{fT}
\end{equation}
where $T_m(x) = \cos (m \arccos x)$ for $x \in [-1, 1]$ is the $m$-th Chebyshev 
polynomial. 

This way, the polynomial approximant $h_{M+N+1}(x)$ of the convolution $h(x)$ 
can be written as the product of a $[-2, 0] \times (N+1)$ column quasi-matrix%%%
\footnote{An $[a, b] \times n $ column quasi-matrix is a matrix with $n$ 
columns, where each column is a univariate function defined on an interval $[a, 
b]$, and can be deemed as a continuous analogue of a tall-skinny matrix, where 
the rows are indexed by a continuous, rather than discrete, variable. For the 
notion of quasi-matrices, see, for example, \cite{ste, bat}.} %%%%%
$\tilde{R}$ and the coefficient vector $\ub=(b_0,\ldots,b_N)^T$
\begin{equation}
h_{M+N+1}(x) = \tilde{R} \ub =
\begin{bmatrix}[c|c|c|c]
 &  &  &  \\
 &  &  &  \\
(f_M{\ast}T_0)(x) & (f_M{\ast} T_1)(x) & \cdots & (f_M{\ast}T_N)(x) \\
 &  &  &  \\
 &  &  &  
\end{bmatrix}
\ub, \label{quasi}
\end{equation}
where the $n$-th column of $\tilde{R}$ is the convolution of $f_M(x)$ and 
$T_n(x)$. Besides the notation we have introduced, asterisks are also used here 
and in the remainder of this paper to denote the convolution of two functions. 
For example, $(f_M\ast T_n)(x)$ is the convolution of $f_M(x)$ and $T_n(x)$. 
For convenience, we will also denote by $\tilde{R}_n$ the $n$-th column of 
$\tilde{R}$, i.e. $(f_M\ast T_n)(x)$, with the index $n$ starting from $0$.

If $h_{M+N+1}(x)$ and $\{(f_M \ast T_n)(x)\}_{n=0}^N$ are translated to 
$[-1, 1]$ by the change of variables $y = x+1$, they can be expressed as 
Chebyshev series of degree $M+N+1$ and $M+n+1$, respectively:
\begin{equation}
\begingroup
\arraycolsep=10pt
\begin{array}{l l}
\begin{aligned}
h_{M+N+1}(y) &= \sum_{k = 0}^{M+N+1} c_k T_k(y), \\
(f_M \ast T_n)(y) &= \sum_{k = 0}^{M+N+1} R_{k, n} T_k(y), ~~~~ 0 \leqslant n 
\leqslant N,
\end{aligned}
\end{array}
\endgroup
\label{chebexpan}
\end{equation}
where $y \in [-1,1]$ and $R_{k, n} =0$ for any $k > M+n+1$. Substituting 
\eqref{chebexpan} into \eqref{quasi}, we have
\begin{equation}
\sum_{k = 0}^{M+N+1} c_k T_k(y) = \sum_{n = 0}^N b_n \hspace{-2mm}\sum_{k = 
0}^{M+N+1} R_{k, n} T_k(y),
\end{equation}
or, equivalently,
\begin{equation}
\uc = R\ub, \label{cRb}
\end{equation}
where $\uc=(c_0,\ldots,c_{M+N+1})^T$ and $R$ is an $(M+N+2) \times (N+1)$ 
matrix that collects the Chebyshev coefficients of $\tilde{R}_n$ for $0 
\leqslant n \leqslant N$. Since $R_{k, n} =0$ for $k > M+n+1$, the lower 
triangular part of $R$ below $(M+1)$-th subdiagonal are zeros (see Figure 
\ref{FIG:unstable}). We shall call $R$ the convolution matrix that approximates 
the convolution operator $V[f]$. What makes $R$ important is the fact that with 
$R$ available either of $\ub$ and $\uc$ can be calculated when the other is 
given. Thus our goal is to calculate $R$ accurately and efficiently.

Thus far the only attempt to approximate the convolution operator in the same 
vein was given by Hale and Townsend \cite{hal}, where they considered the same 
problem but with $f(x)$ and $g(x)$ approximated by Legendre series. Their 
method exploits the fact that the Fourier transform of Legendre polynomial 
$P_m(x)$ is spherical Bessel function $j_m(x)$ with a simple rescaling. They 
also show that $(P_m \ast P_n)(x)$ can be represented as a rescaled inverse 
Fourier transform of the product $j_m(x)j_n(x)$, due to the convolution theorem. 
Based on these results, the entries of a Legendre-based convolution matrix 
$R^{(1/2)}$ are represented as infinite integrals involving triple-product of 
spherical Bessel functions and the three-term recurrence satisfied by spherical 
Bessel functions finally leads to a four-term recurrence relation satisfied by 
the entries of $R^{(1/2)}$. Unfortunately, the Fourier transforms of other 
classical orthogonal polynomials do not have a simple representation in terms of 
spherical Bessel functions or other special functions that enjoy a similar 
recurrence relation \cite{dix}. Therefore, the method in \cite{hal} cannot be 
easily extended to the cases where $f(x)$ and $g(x)$ are approximated using 
other classical orthogonal polynomials, e.g. Chebyshev, and this is what this 
article addresses.

In another work with a similar setting \cite{xu}, spectral approximations of 
convolution operators are constructed when $f(x)$ and $g(x)$ are approximated by 
Fourier extension approximants. It is shown that convolution can be represented 
in terms of products of Toeplitz matrices and coefficient vectors of the Fourier 
extension approximants, based on which an $\bigO(N' (\log N')^2)$ fast algorithm 
is derived for approximating the convolution of compactly supported functions, 
where $N'$ is the number of degrees of freedom in each of the Fourier extension 
approximants for $f(x)$ and $g(x)$.

In an investigation carried out simultaneously \cite{lou}, closed-form formulae 
are derived for the convolution of classical orthogonal polynomials. However, 
this explicit formula is too complicated and numerically intractable to be 
computationally useful for direct construction of $R$.

In this article, we first generalize the recurrence relation found in \cite{hal} 
to Chebyshev-based convolution matrices. Instead of resorting to the Fourier 
transform of Chebyshev polynomials and recurrence of spherical Bessel functions, 
we exploit the three-term recurrence of the derivatives of Chebyshev polynomials 
to show a recurrence relation satisfied by the columns of $\tilde{R}$. Further, 
a five-term recurrence relation satisfied by the entries of $R$ can be obtained 
by replacing columns of $\tilde{R}$ with their Chebyshev coefficients. With 
this recurrence relation and a symmetric property, the entries of $R$ can be 
calculated efficiently and numerically stably, yielding spectral approximations 
to the convolution operators of Volterra type. The accuracy of $R$ and 
its applications are shown by various numerical examples. Finally, we extend 
our approach to broader Jacobi-family orthogonal polynomials, where the results 
of \cite{hal} are covered as a special case.

Our exposition could either begin with the Jacobi-based convolution and treat 
Gegenbauer, Legendre, and Chebyshev as special cases in a cascade, or start with 
Chebyshev, extend to Gegenbauer, and further to Jacobi. We choose the latter, 
since most derivations and proofs are much simpler with Chebyshev polynomials 
and analogues can be easily drawn to others. Also, the Chebyshev-based 
convolution is the most commonly-used in practice and, therefore, deserves a 
more elaborated discussion. 

Our discussion is organized as follows. In Section \ref{SEC:cheb}, we derive 
the recurrence relation satisfied by columns of $\tilde{R}$. In Section 
\ref{SEC:algo}, the recurrence relation for entries of $R$ is derived based on 
that of $\tilde{R}$. We show a stability issue when this recurrence relation is 
naively used for the construction of $R$ and provide a numerically stable 
algorithm by making use of the symmetric structure of $R$. In Section 
\ref{SEC:other}, we extend the results of Sections \ref{SEC:cheb} and 
\ref{SEC:algo} to Gegenbauer- and Jacobi-based convolution matrices. The main 
results of this paper are complemented in Section \ref{SEC:lag} by a brief 
discussion about the approximation of the convolution operators defined by 
functions on $[0, \infty]$ using weighted Laguerre polynomials, before we give 
a few closing remarks in the final section. 

%%%%%%%%%%%%%%%%%%%%%%%%%%%%%%%%%%%%%%%%%%%%%%%%%%%%%%%%%%%%%%%%%%%%%%%%%%%%%%%%

\section{Recurrence satisfied by convolutions of Chebyshev polynomials} 
\label{SEC:cheb}
We start with the following recurrence relation that can be derived from the 
fundamental three-term recurrence relation of Chebyshev polynomials.
\begin{lemma} \label{LEM:recT}
For $n \geqslant 0$, the $n$-th Chebyshev polynomial $T_n(x)$ can be written as 
a combination of the derivatives of $T_{n-1}$ and $T_{n+1}$:
\begin{subequations}
\begin{equation}
\displaystyle T_n(x) =
\begin{cases}
\displaystyle \frac{1}{2(n+1)}\frac{\md T_{n+1}(x)}{\md x} - \frac{1}{2(n-1)} 
\frac{\md T_{|n-1|}(x)}{\md x}, & 
\quad n \neq 1, \\[4mm]
\displaystyle \frac{1}{4}\frac{\md T_2(x)}{\md x}, & \quad n = 1,
\end{cases}
\label{drecT}
\end{equation}
by integrating which we have the recurrence relation of Chebyshev polynomials 
and their indefinite integrals:
\begin{equation}
\displaystyle \int T_n(t)\md t =
\begin{cases}
\displaystyle \frac{T_{n+1}(x)}{2(n+1)} - \frac{T_{|n-1|}(x)}{2(n-1)}, & \quad 
n \neq 1, \\[4mm]
\displaystyle T_2(x)/4, & \quad n = 1.
\end{cases}
\label{irecT}
\end{equation}
\label{recT}
\end{subequations}
\end{lemma}
\begin{proof}
The proof can be found in many standard texts on Chebyshev polynomials. See, 
for example, \cite[p. 32]{mas}.
\end{proof}

Like the convolution of functions defined on the entire real line, 
the convolution operator $V[f]$ also enjoys commutativity:
\begin{equation}
\ixx f(x-t)g(t) \dt = \ixx f(t)g(x-t) \dt, ~~~ x \in [-2, 0],\label{comm}
\end{equation}
which can be shown by a change of variables with $T = x-t$.

Our first main result is the recurrence relation satisfied by the convolutions 
of a Chebyshev series and Chebyshev polynomials, which follows as a consequence 
of \eqref{drecT}.
\begin{theorem}[Recurrence of convolutions of Chebyshev polynomials] 
\label{THM:T}
The convolutions of Chebyshev polynomials and the Chebyshev series $f_M(x)$ 
given in \eqref{fT} recurse as follows:
\begin{subequations}
\begin{equation}
\iy f_M(x-t) T_0(t) \dt = \iy f_M(t) \dt, \label{T0}
\end{equation}
\begin{equation}
\iy f_M(x-t) T_1(t) \dt = \iiy f_M(t) \dd - \iy f_M(t) \dt, \label{T1}
\end{equation}
\begin{equation}
\iy f_M(x-t) T_2(t) \dt = 4\iiy f_M(x-t) T_1(t) \dd + \iy f_M(t) \dt, 
\label{T2}
\end{equation}
and for $n \geqslant 2$,
\begin{equation}
\begin{multlined}
\iy f_M(x-t) T_{n+1}(t) \dt = 2(n+1)\iiy f_M(x-t) T_n(t) \dd \\
+\frac{n+1}{n-1} \iy f_M(x-t) T_{n-1}(t) \dt + \frac{2(-1)^n}{n-1} \iy 
f_M(t)\dt,
\end{multlined}
\label{Tn}
\end{equation}
\label{T}
\end{subequations}
where $x \in [-2, 0]$ and $y = x+1 \in [-1, 1]$. 
\end{theorem}

\begin{proof}
Let us first show \eqref{Tn} by differentiating $\ixx T_{n+1}(x-t) T_m(t) \dt$ 
with respect to $x$. The Leibniz integral rule gives
\begin{equation}
\frac{\md}{\md x} \hspace{-1mm} \ixx \hspace{-2mm} T_{n+1}(x-t) 
T_m(t) \dt = \ixx \hspace{-1mm} \frac{\md T_{n+1}(x-t)}{\md x} 
T_m(t) \dt + (-1)^{n+1}T_m(x+1), \label{t1_np1}
\end{equation}
where $T_{n+1}(-1) = (-1)^{n+1}$ is used. Similarly, we have
\begin{equation}
\frac{\md}{\md x} \hspace{-1mm} \ixx \hspace{-2mm} T_{n-1}(x-t) 
T_m(t) \md t = \ixx \hspace{-1mm} \frac{\md T_{n-1}(x-t)}{\md x} 
T_m(t) \dt + (-1)^{n-1}T_m(x+1). \label{t1_nm1}
\end{equation}
By combining \eqref{drecT}, \eqref{t1_np1}, and \eqref{t1_nm1}, we have
\begin{equation}
\begin{multlined}
\frac{\md}{\md x} \hspace{-1mm} \ixx T_{n+1}(x-t) T_m(t) \dt = 
2(n+1) \ixx T_n(x-t)T_m(t)\dt\\
+\frac{n+1}{n-1} \frac{\md}{\md x} \hspace{-1mm} \ixx T_{n-1}(x-t) 
T_m(t) \dt + \frac{2(-1)^n}{n-1}T_m(x+1).
\end{multlined}
\end{equation}
Noting that all the terms are polynomials of $x+1$, we integrate with respect to 
$x+1$ from $-1$ to an arbitrary $y \in [-1,1]$ to get rid of the derivatives:
\begin{equation*}
\begin{multlined}
\left[ \ixx T_{n+1}(x-t) T_m(t) \dt \right]_{x+1=-1}^{x+1=y} = 2(n+1)\iiy 
T_n(x-t) T_m(t) \dd \\
+\frac{n+1}{n-1} \left[ \ixx T_{n-1}(x-t) T_m(t) \dt \right]_{x+1=-1}^{x+1=y} + 
\frac{2(-1)^n}{n-1}\iy T_m(t) \dt.
\end{multlined}
\end{equation*}
Since $\ixx T_{n\pm 1}(x-t) T_m(t) \dt$ vanishes at $x = -2$, the last equation 
becomes
\begin{equation}
\begin{multlined}
\iy T_{n+1}(x-t) T_m(t) \dt = 2(n+1)\iiy T_n(x-t) T_m(t) \dd \\
+\frac{n+1}{n-1} \iy T_{n-1}(x-t) T_m(t) \dt + \frac{2(-1)^n}{n-1}\iy T_m(t) 
\dt.
\end{multlined} 
\label{T_interm}
\end{equation}
Here, we have intentionally left the variable $x$ in the integrands of the 
first two single integrals without replacing it by $y-1$ in order to keep the 
integrands neat.

By the commutativity \eqref{comm}, we are free to swap the arguments $x-t$ 
and $t$ in all convolutions to have
\begin{equation}
\begin{multlined}
\iy T_m(x-t) T_{n+1}(t) \dt = 2(n+1)\iiy T_m(x-t) T_n(t) \dd\\
+\frac{n+1}{n-1} \iy T_m(x-t) T_{n-1}(t)\dt+\frac{2(-1)^n}{n-1}\iy T_m(t)\dt.
\end{multlined}
\end{equation}
Finally, \eqref{Tn} is obtained by linearity.

We can show \eqref{T1} and \eqref{T2} similarly and obtain \eqref{T0} by noting 
that $T_0(x) = 1$.
\end{proof}

Theorem \ref{THM:Rn} reveals a recurrence relation satisfied by the columns of 
$\tilde{R}$. To see this, we replace $\ixx f_M(x-t) T_n(t) \dt$ or $\iy f_M(x-t) 
T_n(t) \dt$ in \eqref{T} by the much compacter notation $\tilde{R}_n(y)$ to have
\begin{subequations}
\begin{empheq}{alignat=2}
\tilde{R}_0(y) &= \iy f_M(t) \dt, \label{TTR0}\\[-1mm]
\tilde{R}_1(y) &= \iy \tilde{R}_0(\ty) \dty - \tilde{R}_0(y), 
\label{TTR1}\\[-1mm]
\tilde{R}_2(y) &= 4\iy \tilde{R}_1(\ty)\dty + \tilde{R}_0(y), 
\label{TTR2}\\[-1mm]
\tilde{R}_{n+1}(y) &= 2(n+1)\iy \tilde{R}_n(\ty) \dty +\frac{n+1}{n-1} 
\tilde{R}_{n-1}(y) + \frac{2(-1)^n}{n-1}\tilde{R}_0(y) ~~~(n \geqslant 2). 
\label{TTRn}
\end{empheq}
\label{TTR}
\end{subequations}

Of course, the terms in \eqref{TTR} are continuous functions of $y$ and would 
not be useful for numerical computing until they are fully discretized.
%%%%%%%%%%%%%%%%%%%%%%%%%%%%%%%%%%%%%%%%%%%%%%%%%%%%%%%%%%%%%%%%%%%%%%%%%%%%%%%%

\section{Constructing the convolution matrices} \label{SEC:algo}

In this section, we show the discrete counterpart of the recurrence relation 
\eqref{T} or \eqref{TTR} based on which the convolution matrix $R$ can be 
constructed. We begin with the integration of a Chebyshev series.
\begin{lemma}[Indefinite integral of a Chebyshev series]\label{LEM:int}
For Chebyshev series $\phi(x) = \sum_{j=0}^J \alpha_j T_j(x)$ with 
$x \in [-1, 1]$, its indefinite integral, when expressed in terms of Chebyshev 
polynomials, is
\begin{equation*}
\int_{-1}^x \phi(t) dt = \sum_{j=0}^{J+1} \tilde{\alpha}_j T_j(x),
\end{equation*}
with the coefficients
\begin{subequations}
\begin{equation}
\displaystyle \tilde{\alpha}_j =
\begin{cases}
\displaystyle \frac{\alpha_{j-1}-\alpha_{j+1}}{2j}, & \quad 2 \leqslant j 
\leqslant J+1,\\[3mm]
\displaystyle \alpha_{0}-\frac{\alpha_{2}}{2}, & \quad j = 1, \\[1mm]
\displaystyle \sum_{k=1}^{J+1} (-1)^{k+1}\tilde{\alpha}_k , & \quad j = 0, \\
\end{cases}
\label{antideriv}
\end{equation}
\end{subequations}
where $\alpha_{J+1} = \alpha_{J+2} = 0$.
\end{lemma}
\begin{proof}
This is a straightforward result of \eqref{irecT}. A slightly different version 
of this lemma can be found in \cite[\S 2.4.4]{mas}.
\end{proof}

Now we have all the ingredients for computing the entries of $R$. By 
\eqref{TTR0}, $R_{:,0}$ are the Chebyshev coefficients of the indefinite 
integral of $f_M(x)$ subject to $\tilde{R}_0(-1) = 0$. We state this as a 
theorem with the proof omitted.

\begin{theorem}[Construction of the zeroth column of $R$] \label{THM:R0}
The entries of the zeroth column of $R$ are
\begin{equation}
R_{k,0} =
\begin{cases}
\displaystyle 0 &  k > M+1, \\[1mm]
\displaystyle \frac{a_{k-1}-a_{k+1}}{2k} & 2 \leqslant k \leqslant M+1,\\[3mm]
\displaystyle a_0-\frac{a_2}{2} & \quad k = 1, \\[1mm]
\displaystyle \sum_{j=1}^{M+1} (-1)^{j+1}R_{j,0} & \quad k = 0,
\end{cases}
\label{R0}
\end{equation}
with $a_{M+1} = a_{M+2} = 0$.
\end{theorem}

With the zeroth column of $R$, we can recurse for the subsequent columns as 
suggested by \eqref{TTR1}, \eqref{TTR2}, and \eqref{TTRn}.
\begin{theorem}[Recurrence of columns of $R$] \label{THM:Rn}
For $1 \leqslant k \leqslant M+N$, the entries 
of $R$ have the following recurrence relation:
\begin{subequations}
\begin{equation}
R_{k, 1}=-R_{k, 0}+\frac{1}{2k}R'_{k-1, 0}-\frac{1}{2k}R_{k+1, 0} \label{R1}
\end{equation}
\begin{equation}
R_{k, 2}=R_{k, 0}+\frac{2}{k}R'_{k-1, 1}-\frac{2}{k}R_{k+1, 1} \label{R2}
\end{equation}
and when $n \geqslant 2$,
\begin{equation}
R_{k, n+1}=\frac{2(-1)^n}{n-1}R_{k, 0}+\frac{n+1}{n-1}R_{k, n-1} + 
\frac{n+1}{k} R'_{k-1, n}-\frac{n+1}{k}R_{k+1, n}, \label{Rn}
\end{equation}
where the prime denotes that the coefficient of the term is doubled when $k=1$. 

For any $n \geqslant 1$,
\begin{equation}
R_{0, n} = \sum_{j=1}^{M+1} (-1)^{j+1}R_{j, n}. \label{constants}
\end{equation}
\label{TR}
\end{subequations}
\end{theorem}
\begin{proof}
Substituting into \eqref{TTR1} the Chebyshev series of $\tilde{R}_0(y)$ and 
$\tilde{R}_1(y)$ gives
\begin{equation}
\sum_{k=0}^{M+N+1} R_{k, 1} T_k(y) = \int_{-1}^{y}\left( \sum_{k=0}^{M+N+1} 
R_{k, 0} T_k(\ty) \right) \dty - \hspace{-3mm} \sum_{k=0}^{M+N+1} R_{k, 0} 
T_k(y),
\end{equation}
where $R_{k,0} = 0$ for $k>M+1$ and $R_{k,1} = 0$ for $k>M+2$. Replacing the 
integral term by its Chebyshev series using Lemma \ref{LEM:int} and matching the 
$T_k(y)$ terms for each $k \geqslant 1$ gives \eqref{R1}.

The recurrence relations \eqref{R2} and \eqref{Rn} can be derived similarly 
from \eqref{TTR2} and \eqref{TTRn}, respectively. The entries in the zeroth row 
are set using \eqref{constants} so that $\tilde{R}_{n}(-1) = 0$.
\end{proof}

\begin{figure}[t]
\centering
\includegraphics[scale=0.7]{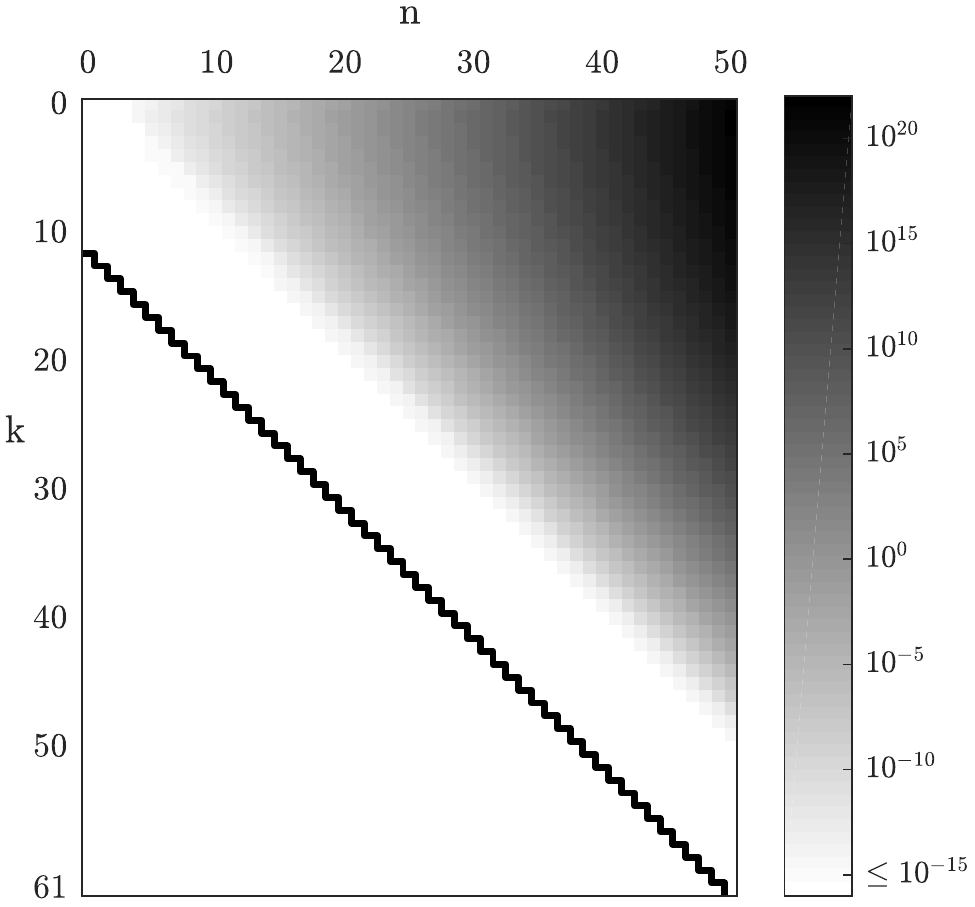}
\caption{Error growth when recursing using \eqref{Rn}. The entries of $R$ below 
the dark line are exactly zero.}
\label{FIG:unstable}
\end{figure}

The calculation of $R$ could have been as easy as suggested by Theorem 
\ref{THM:Rn}: calculate the zeroth column of $R$ following \eqref{R0} and 
recurse using \eqref{R1}, \eqref{R2}, and \eqref{Rn}. Unfortunately, \eqref{Rn} 
is not numerically stable even in absolute sense. 

\textbf{Example 1:} To see the instability, we take a randomly generated 
Chebyshev series of degree $10$, i.e. $f_{M} = \sum_{m=0}^{10} a_m T_m(x)$, with 
$|a_m| \leqslant 1$ and compare the entries in columns $1$ to $50$ calculated
recursively using Theorem \ref{THM:Rn} with the exact values computed 
symbolically using \textsc{Mathematica}. Figure \ref{FIG:unstable} shows 
the entrywise absolute error. The error in the entries above the main diagonal 
grows very rapidly, which is similar to what is observed in \cite{hal} for the 
Legendre case. In fact, the rounding errors in $R_{k-1, n}$ and $R_{k+1, n}$ 
in \eqref{Rn} are subject to an amplification by the factor $(n+1)/k$, which is 
larger than $1$ above the main diagonal. The recursion snowballs the errors 
introduced in each use of \eqref{Rn} very quickly, resulting in the computed 
values soon to become totally garbage. In the worst scenario, an error could be 
magnified $n$ folds in the $n$-th recursion and blows up at a rate of factorial. 
For instance, the absolute error in the entry at the top right corner of $R$ in 
Figure \ref{FIG:unstable}, i.e. $R_{0, 50}$, is $\bigO(10^{20})$, while the true 
value is about $10^{-3}$ in this example.

Indeed, \eqref{Rn} is only useful for calculating the entries on and below the 
main diagonal, that is, the entries in the region labeled by $A$ in Figure 
\ref{FIG:cheb_structure}.
To circumvent the instability, we make the following critical observation which 
is similar to the one made in \cite{hal} for the Legendre-based convolution 
matrices (see Section \ref{SEC:leg}).

\begin{theorem}[Symmetry of $R$] \label{THM:symT}
For $M+1 \leqslant k,n \leqslant N$,
\begin{equation}
R_{n,k} = \frac{(-1)^{n+k}k}{n}R_{k,n}. \label{symT}
\end{equation}
\end{theorem}

We find it easy to prove Theorem \ref{THM:symT} by deducing it from the 
analogous result of the Jacobi-based convolution matrices and, therefore, 
defer the proof to Section \ref{SEC:jac}.

\begin{figure}[t]
\centering
\includegraphics[scale=0.3]{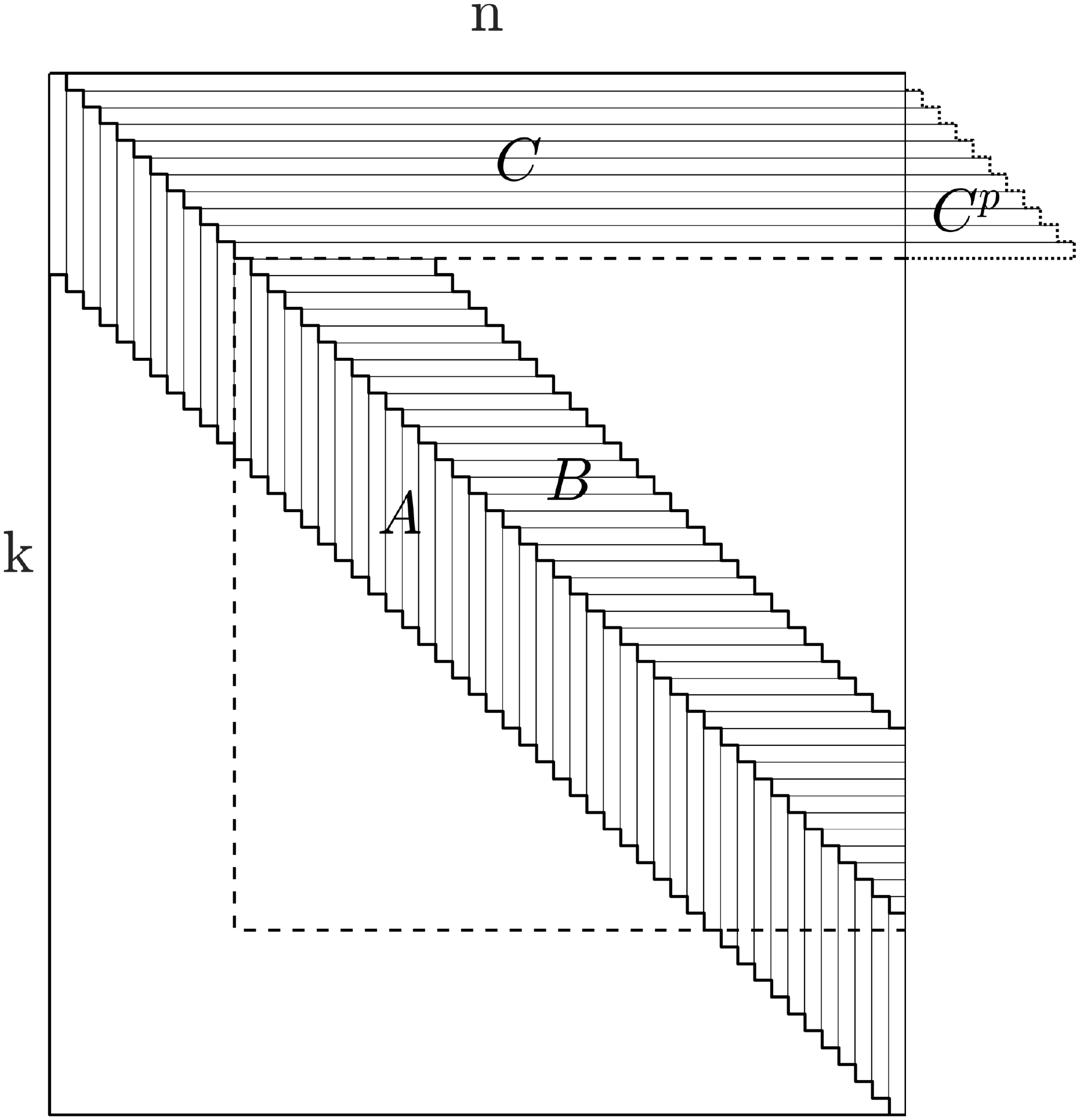}
\caption{The convolution matrix $R$ defined by an $M$-Chebyshev series is an 
almost-banded matrix with bandwidth $M+1$ plus $M+1$ rows at the top. 
The construction of $R$ starts in region $A$ (the main diagonal and the first 
$M+1$ subdiagonals) by recursion from the left to the right using \eqref{R1}, 
\eqref{R2}, and \eqref{Rn}. The entries in region $B$ (the first $M+1$ 
superdiagonals), are obtained by symmetry \eqref{symT}. With the padding region 
$C^p$, the recast recurrence relation \eqref{RnC} allows a stable recursion for 
entries in region $C$ (the top $M+1$ rows). All remaining entries are exact 
zeros.} \label{FIG:cheb_structure}
\end{figure}

Theorem \ref{THM:symT} shows the symmetry of $R$ up to a scaling factor, apart 
from the top and the bottom $M+1$ rows and the first $M+1$ columns. In Figure 
\ref{FIG:cheb_structure}, the symmetric part of $R$ is marked by the dashed 
lines. The important implications of this symmetry are (1) this $(N-M) \times 
(N-M)$ symmetric submatrix of $R$ is banded with bandwidth $M+1$; (2) the 
entries in region $B$ can be obtained stably and cheaply by rescaling their 
mirror images about the main diagonal; (3) the entries of the top $M+1$ rows of 
$R$, i.e. region $C$, can then be calculated by the same recurrence relation 
given by \eqref{Rn}.

When \eqref{Rn} is used to calculate the entries in the top $M+1$ rows, we 
rewrite it so that calculation is done by rows, going upward from the bottom of 
region $C$ to the top:
\begin{equation}
R''_{k-1, n}=-\frac{2k(-1)^n}{n^2-1}R_{k, 0} - \frac{k}{n-1}R_{k, n-1} 
+ \frac{k}{n+1} R_{k, n+1}+R_{k+1, n}, \label{RnC}
\end{equation}
where the double prime indicates that the term is halved when $k=1$. This new 
recurrence relation is numerically stable in region $C$. In contrast to 
\eqref{Rn}, the rounding errors are now premultiplied by $k/(n-1)$ or 
$k/(n+1)$, which are less than or, at most, equal to $1$ above the main 
diagonal. Therefore, the rounding errors are diminished in the course of 
recursion, rather than amplified.

It is worth noting that when recursing for the top $M+1$ rows using \eqref{RnC}, 
we have to start with entries beyond the first $N+1$ columns of $R$, so that all 
entries in the ``domain of dependence'' of the zeroth row are counted on. This 
suggests a triangle-shaped padding region, labeled by $C^p$ in Figure 
\ref{FIG:cheb_structure}.

We summarize the stable algorithm described above as follows:
\begin{algorithm}[h!]
\caption{Construction of the convolution matrix $R$} 
\label{ALGO:cheb}
% \begin{multicols}{2}
\begin{algorithmic}[1]
\STATE Construct the non-zero entries in the zeroth column $R_{:, 0}$ using 
\eqref{R0}.
\STATE Calculate the non-zero entries on and below the main diagonal (Region A) 
using \eqref{R1}, \eqref{R2}, and \eqref{Rn}.
\STATE Calculate the non-zero entries above the main diagonal (Region B) in 
rows $M+1$ to $N$ using \eqref{symT}.
\STATE Recurse for the entries above the main diagonal in the top $M+1$ rows 
(Region C) using \eqref{RnC}. 
\end{algorithmic}
% \columnbreak
% \begin{algorithmic}
% \STATE O(x)
% \end{algorithmic}
% \end{multicols}
\end{algorithm}

\textbf{Example 2:} Now we re-compute the same $R$ in Example 1 using Algorithm 
\ref{ALGO:cheb} and plot the entrywise absolute error in Figure \ref{FIG:cheb1}. 
With the stabilized algorithm, the maximum error across all the entries is now 
$2.12 \times 10^{-16}$ in this example. 

\textbf{Example 3:} In Figure \ref{FIG:cheb2}, we show a similar example with 
$M=1000$ and $N=5000$. Again, $\ua$ is generated randomly with $|a_m|\leqslant 
1$. The largest entrywise error across all the entries of $R$ is $1.28 \times 
10^{-15}$. 

\begin{figure}[t]
\centering
\begin{subfigure}[b]{0.45\textwidth}
\hspace{5mm}\includegraphics[scale=0.51]{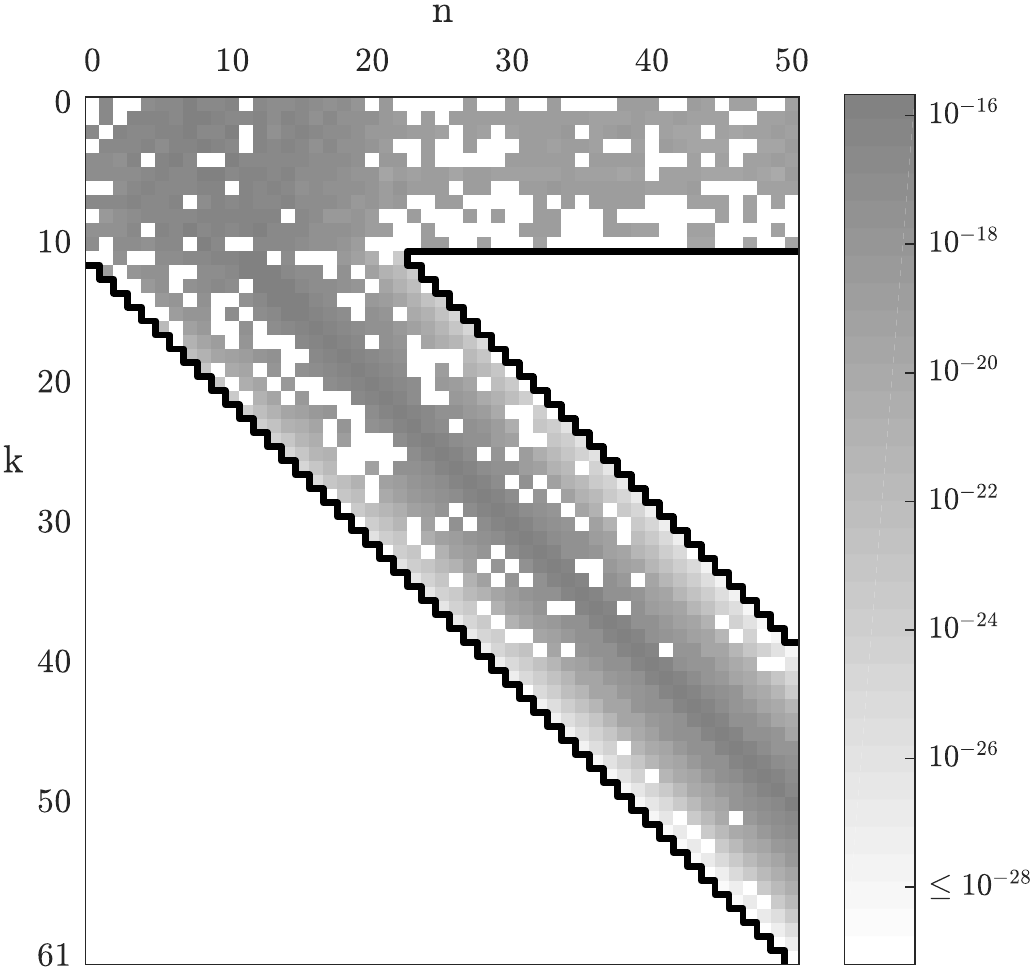}\caption{
$M=10$ , $N=50$.} \label{FIG:cheb1}
\end{subfigure}
\begin{subfigure}[b]{0.45\textwidth}
\hspace{3mm}\includegraphics[scale=0.5]{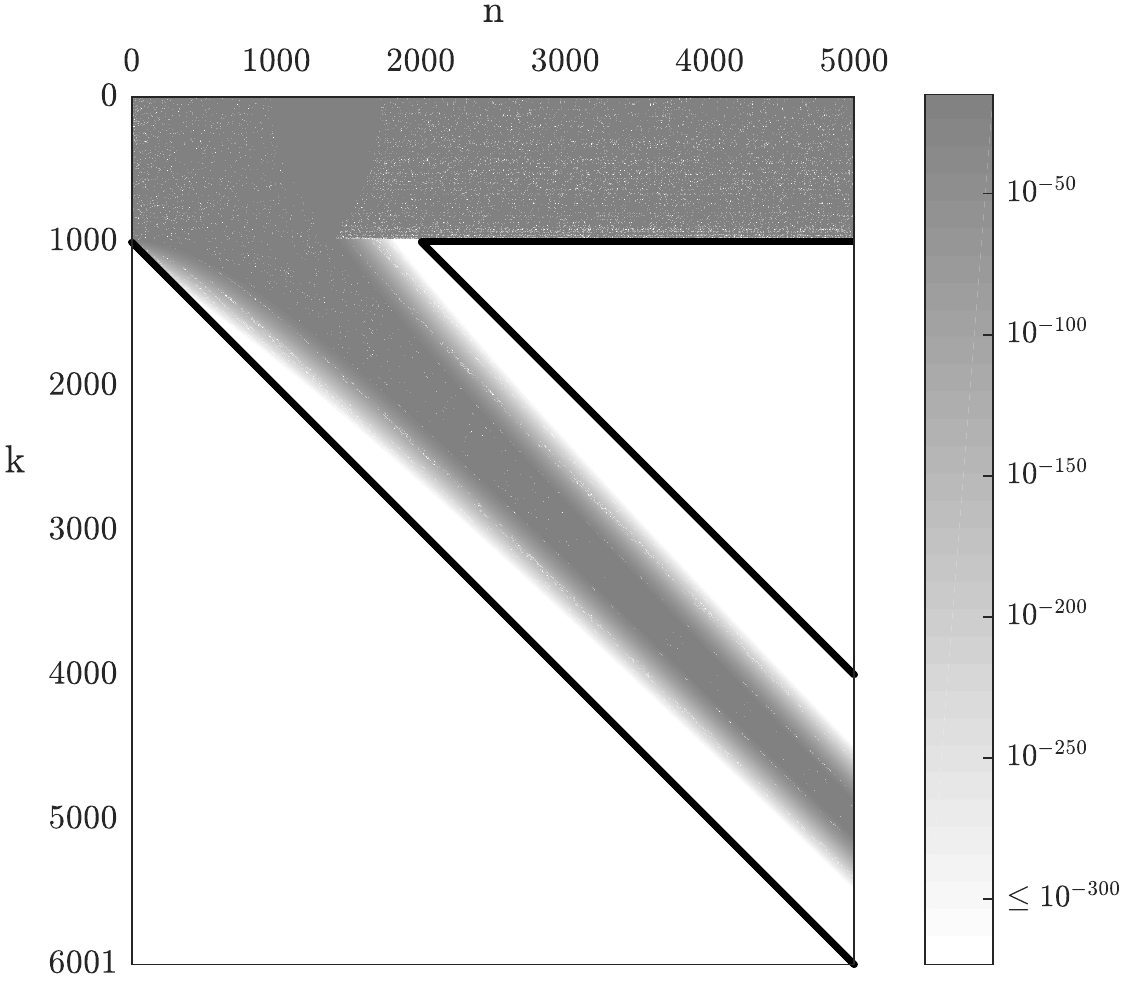}\caption{
$M=1000$, $N=5000$.}\label{FIG:cheb2}
\end{subfigure}
\caption{Entrywise error of the nonzero entries of $R$ obtained using Algorithm 
\ref{ALGO:cheb}. The dark lines circumscribe the region outside which entries 
are exact zeros.}\label{FIG:cheb}
\end{figure}

A curious observation we made in the last two examples and other experiments 
that we carried out is that the magnitudes of the entries in a convolution 
matrix have an enormous range of orders. In Example 2, even though the entries 
of $\ua$ are all $\bigO(10^{-1})$, the exact values obtained symbolically using 
\textsc{Mathematica} show that some entries of $R$ can be as small as
$\bigO(10^{-14})$. Therefore, it makes little sense to talk about the relative 
error of the computed entries, since we cannot expect to be able to compute 
$\bigO(10^{-14})$ values accurately in a relative sense by using 
$\bigO(10^{-1})$ data in floating point arithmetic. 

In Example 3, the magnitudes of the entries vary from $\bigO(10^{-1})$ to 
$\bigO(10^{p})$ with $p < -324$. The most minuscule entries are not even 
representable by the IEEE floating point arithmetic\footnote{The smallest 
subnormal number in the current IEEE floating point standard is $2 \times 
10^{-1074} \approx 4.94 \times 10^{-324}$.} \cite{ieee}. This also suggests that 
we should confine our discussion to absolute accuracy only.

Although the gargantuan discrepancy in the magnitudes of the entries denies any 
attempt to compute them accurately in a relative sense, the convolution matrices 
constructed using our stable algorithm give accurate approximations to 
the convolution operators in the absolute sense and work perfectly fine when 
used for calculating $h_{M+N+1}(x)$ or solving convolution integral equations, 
since it is also only sensible to discuss absolute accuracy in these cases. 

We close this section with a classic example from the renewal theory 
\cite[Example 1.4.3]{bru}, \cite{fel}. 

\textbf{Example 4:} It can be shown that the Volterra convolution integral 
equation
\begin{equation}
u(x) = f(x) + \int_0^x f(x-t) u(t) \: dt, \qquad x \in [0, 2],
\label{renewal}
\end{equation}
where the \textit{convolution kernel}
\begin{equation}
f(x) = \frac{1}{2} x^2 e^{-x},
\end{equation}
has a unique solution
\begin{equation}
u(x) = \frac{1}{3} - \frac{1}{3} \left(\cos\frac{\sqrt{3}}{2}x + \sqrt{3} 
\sin\frac{\sqrt{3}}{2}x\right)e^{-3x/2}.
\end{equation}

We first test the use of convolution matrices as a means of 
approximating the convolution function $h(x)$. To do so, we approximate $f(x)$ 
and $u(x)$ by Chebyshev series $f_M(x)$ and $u_N(x)$ of degrees $16$ and $17$, 
respectively, to machine precision uniformly on $[0, 2]$%%%%%%
\footnote{These optimal degrees are determined using the adaptive chopping 
algorithm \cite{aur} of \textsc{Chebfun} \cite{dri}.}%%%%%%
and then form the convolution matrix $R$ using the Chebyshev coefficients 
$\uc^f$ of $f_{M}(x)$. The product of $R$ and $u_N(x)$'s Chebyshev 
coefficients $\uc^u$ returns us the Chebyshev coefficients of $h_{M+N+1}(x) = 
(f*u)(x)$, which should be a good approximation of $u(x)-f(x)$. Indeed, the 
pointwise absolute error in $h_{M+N+1}(x)$ is displayed in Figure 
\ref{FIG:renewal1}, where the largest error is approximately $1.10 \times 
10^{-16}$.

Next, we take $u(x)$ as unknown and solve for $\uc^u$ with the knowledge of 
$f_M(x)$. This examines the use of convolution matrices in solving convolution 
integral equations. We construct the $(N+18)\times (N+1)$ convolution matrix 
for $N = 1, 3, 5, \ldots, 25$ and denote by $R^N$ the square matrix formed by 
the first $N+1$ rows. Solving
\begin{equation*}
(I - R^N) \uc^u = \uc^{fN}
\end{equation*}
gives us $\uc^u$, where $\uc^{fN}$ is a tailored version of $\uc^f$, either 
by truncation or zero-padding so that it is of length $N+1$. Figure 
\ref{FIG:renewal2} shows the maximum pointwise error of $u_N(x) = \sum_{n=0}^N 
c^u_nT_n(x)$ for increasing $N$. What we see is a \textit{spectral} convergence 
as the size of discretization increases. When $N=17$, the largest pointwise 
error in $[0, 2]$ decays to $1.39 \times 10^{-16}$, effectively of machine 
precision.

\begin{figure}[t]
\centering
\begin{subfigure}[b]{0.45\textwidth}
\includegraphics[scale=0.45]{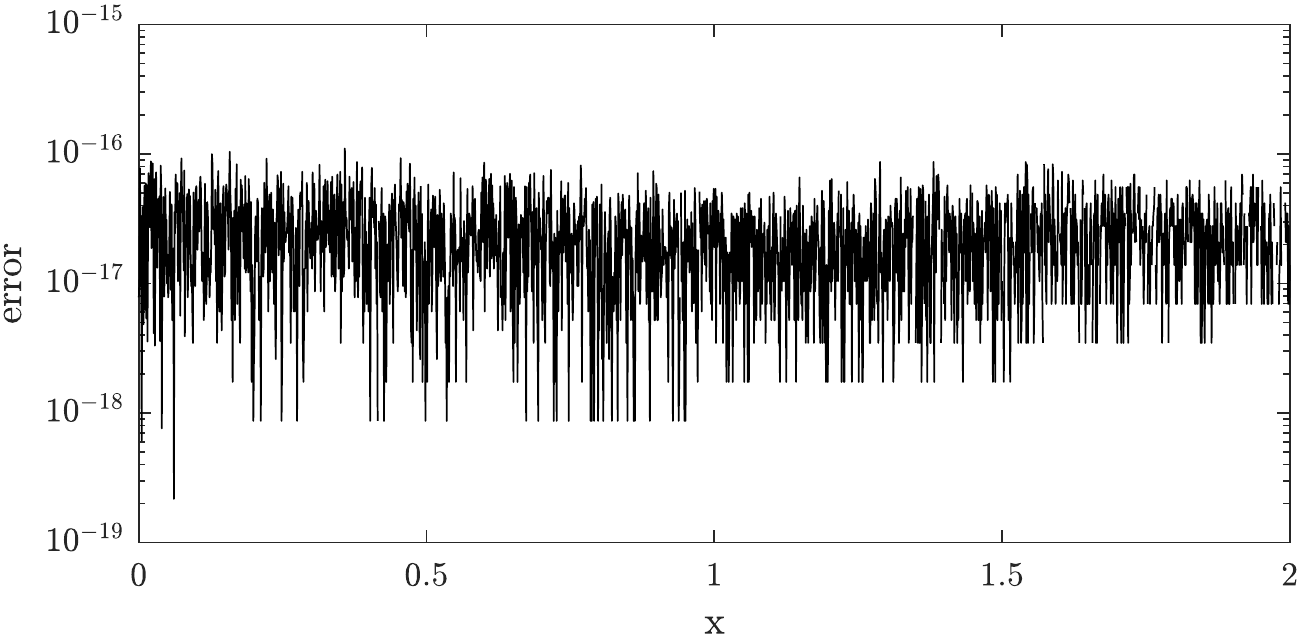}\caption{} 
\label{FIG:renewal1}
\end{subfigure}
~~
\begin{subfigure}[b]{0.45\textwidth}
\includegraphics[scale=0.45]{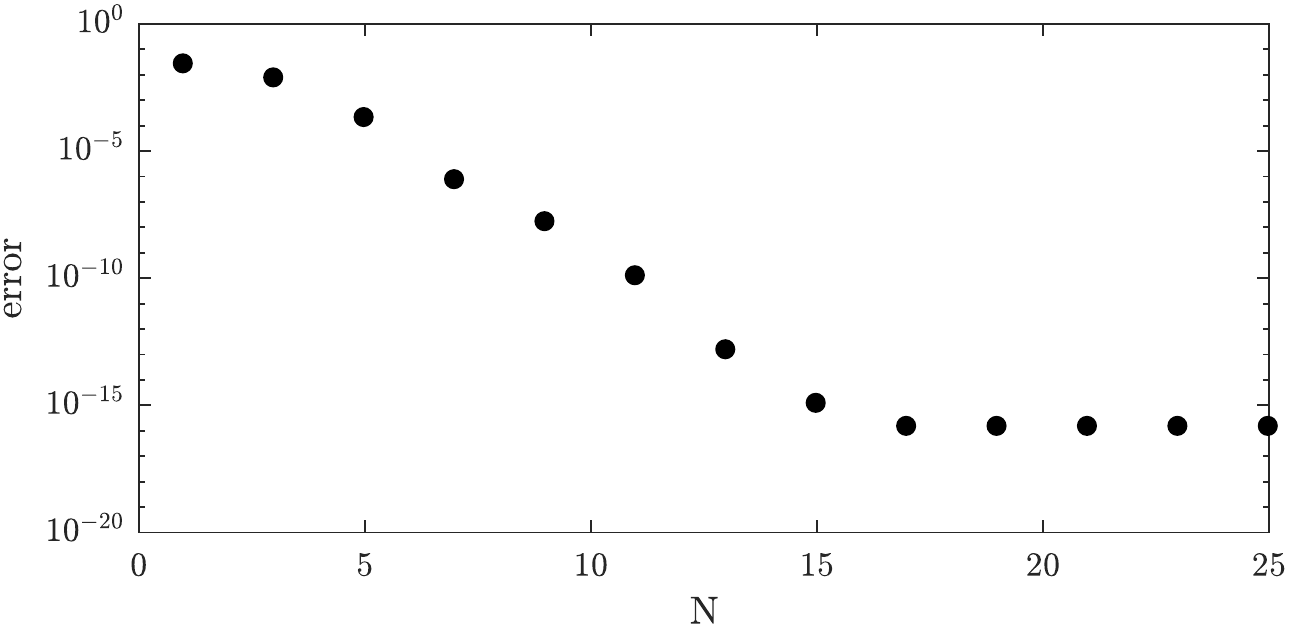}\caption{}\label{FIG:renewal2}
\end{subfigure}
\caption{An example from renewal theory: (a) The pointwise error of the 
computed Chebyshev approximant to the convolution integral on the right-hand 
side of \eqref{renewal}. (b) Spectral convergence of the computed approximant 
to $u(x)$ in \eqref{renewal} using convolution matrices of increasing sizes.}
\label{FIG:renewal}
\end{figure}

%%%%%%%%%%%%%%%%%%%%%%%%%%%%%%%%%%%%%%%%%%%%%%%%%%%%%%%%%%%%%%%%%%%%%%%%%%%%%%%%

\section{Other classical orthogonal polynomials} \label{SEC:other}
To derive the recurrence relations in Sections \ref{SEC:cheb} and 
\ref{SEC:algo}, we have only used the properties of Chebyshev polynomials that 
are also shared by other classical orthogonal polynomials. It is, therefore, 
natural to see how the results in the last two sections extend to Gegenbauer and 
Jacobi spaces. For convergence theory of Gegenbauer and Jacobi approximants, 
see, for example, \cite{wan2, wan, zha}.

\subsection{Convolution matrices in Gegenbauer space} \label{SEC:gegen}
Gegenbauer polynomials $C_n^{(\lambda)}(x)$, also known as ultraspherical 
polynomials, can be defined using the three-term recurrence relation 
\cite[\S 4.7]{sze}
\begin{equation}
2(n+\lambda) xC_{n}^{(\lambda)}(x) = (n+1) C_{n+1}^{(\lambda)}(x) +
(n+2\lambda-1) C_{n-1}^{(\lambda)}(x) \label{3termC}
\end{equation}
with $C_{-1}^{(\lambda)}(x)=0$ and $C_{0}^{(\lambda)}(x)=1$, under the 
constraints $\lambda > -1/2$ and $\lambda\neq 0$. The following lemma, parallel 
to Lemma \ref{LEM:recT}, can be derived using \eqref{3termC}.
\begin{lemma}
For any integer $n$, Gegenbauer polynomials satisfy the following recurrence 
relation that can be written in derivative or integral forms:
\begin{equation}
\begingroup
\arraycolsep=10pt
\begin{array}{l}
\begin{aligned}
2(n+\lambda) C_n^{(\lambda)}(x) = \frac{\md}{\md x} 
\left(C_{n+1}^{(\lambda)}(x) - C_{n-1}^{(\lambda)}(x)\right), \\
\int C_n^{(\lambda)}(t)\md t = \frac{1}{2(n+\lambda)} 
\left(C_{n+1}^{(\lambda)}(x)-C_{n-1}^{(\lambda)}(x)\right),
\end{aligned}
\end{array}
\endgroup
\label{recC}
\end{equation}
where $C^{\lambda}_{n}(x) = 0$ for $n < 0$.
\end{lemma}
\begin{proof}
See \cite[\S 4.7]{sze}.
\end{proof}

We omit the proofs for the next two theorems as they are analogous to those of 
Theorems \ref{THM:T}, \ref{THM:R0}, and \ref{THM:Rn}.
\begin{theorem}[Recurrence of convolutions of Gegenbauer polynomials] 
\label{THM:C}
For Gegenbauer series $f_M(x) = \sum_{m=0}^M a_m C^{(\lambda)}_m(x)$,
\begin{equation}
\iy f_M(x-t) C_0^{(\lambda)}(t) \dt = \iy f_M(t) \dt \label{C0}
\end{equation}
and the convolutions of $f_M(x)$ and $C^{(\lambda)}_n(x)$ recurse:
\begin{equation}
\begin{multlined}
\iy f_M(x-t) C_{n+1}^{(\lambda)}(t)\dt = 2(n+\lambda) \iiy f_M(x-t) 
C_n^{(\lambda)}(t) \dd \\
+\iy f_M(x-t) C_{n-1}^{(\lambda)}(t)\dt + S^{(\lambda)}_n \iy f_M(t)\dt,
\end{multlined}\label{Cn}
\end{equation}
where $x \in [-2, 0]$, $y = x+1 \in [-1, 1]$ and
\begin{equation}
S^{(\lambda)}_n =\frac{2 (-1)^{n+1} (\lambda +n)(2\lambda-1)_n}{(n+1)!},
\label{SC}
\end{equation}
where $(\cdot)_n$ is the Pochhammer symbol for ascending factorial with $(a)_n 
\coloneqq a(a+1)(a+2)\cdots (a+n-1)$ and $(a)_0\coloneqq1$.
\end{theorem}

\begin{theorem}[Construction of $R^{(\lambda)}$] \label{THM:CR}
The entries of the zeroth column of $R^{(\lambda)}$ are
\begin{subequations}
\begin{equation}
R^{(\lambda)}_{k,0} =
\begin{cases}
0 & \quad k > M+1, \\[2mm]
\displaystyle \frac{a_{k-1}}{2(k+\lambda-1)}-\frac{a_{k+1}}{2(k+\lambda+1)} & 
\quad 1 \leqslant k \leqslant M+1, \\[4mm]
\displaystyle \sum_{j=1}^{M+1} (-1)^{j+1}\frac{(2\lambda)_j}{j!} 
R^{(\lambda)}_{j,0} & \quad k = 0,
\end{cases}
\label{CR0}
\end{equation}
with $a_{M+1} = a_{M+2} = 0$. For $n \geqslant 0$,
\begin{equation}
R^{(\lambda)}_{k, n+1}= S^{(\lambda)}_nR^{(\lambda)}_{k, 0}+R^{(\lambda)}_{k, 
n-1} + \frac{n+\lambda}{k+\lambda-1} R^{(\lambda)}_{k-1, n} 
-\frac{n+\lambda}{k+\lambda+1}R^{(\lambda)}_{k+1, n},
\label{CRn}
\end{equation}
\end{subequations}
where $R^{(\lambda)}_{:, -1}$ are understood to be zeros.
\end{theorem}

Similar to the Chebyshev case, the Gegenbauer-based convolution matrices 
are also almost-banded with a symmetric submatrix. 
\begin{theorem}[Symmetry of $R^{(\lambda)}$] \label{THM:symC}
For $M+1 \leqslant k,n \leqslant N$,
\begin{equation}
R^{(\lambda)}_{k,n} = (-1)^{k+n}\frac{k+\lambda}{n+\lambda}R^{(\lambda)}_{n,k}. 
\label{symC}
\end{equation}
\end{theorem}
Again, we defer the proof to Section \ref{SEC:jac}. 

Same stability issue occurs if Gegenbauer-based convolution matrices are 
constructed naively using \eqref{CRn}. The stable algorithm is given in 
Algorithm \ref{ALGO:geja}. Analogously, we need to recast \eqref{CRn} before it 
can be used for calculating the entries above the main diagonal:
\begin{equation}
R^{(\lambda)}_{k-1, n}= \frac{k+\lambda-1}{n+\lambda} \left(-S^{(\lambda)}_n 
R^{(\lambda)}_{k, 0}+R^{(\lambda)}_{k, n+1} - R^{(\lambda)}_{k,n-1} \right) 
+\frac{k+\lambda-1}{k+\lambda+1}R^{(\lambda)}_{k+1, n}. \label{CRnC}
\end{equation}

\subsubsection{Convolution matrices in Legendre space} \label{SEC:leg}
Gegenbauer polynomials reduce to Legendre polynomials $P_n(x)$ when $\lambda 
= 1/2$, i.e. $P_n(x) = C_n^{(1/2)}(x)$. In this case,  $S^{(1/2)}_n = 0$ , 
annihilating the last term in \eqref{Cn} and the first term on the right-hand 
side of \eqref{CRn}. Now \eqref{CRn} reduces to a four-term recurrence relation
\begin{equation}
R^{(1/2)}_{k, n+1}= R^{(1/2)}_{k, n-1} + \frac{2n+1}{2k-1} R^{(1/2)}_{k-1, n} - 
\frac{2n+1}{2k+3} R^{(1/2)}_{k+1, n}, \label{LR}
\end{equation}
which is exactly the one found in \cite{hal} via spherical Bessel functions. 

In the absence of the $R^{(\lambda)}_{k, 0}$ term, symmetry \eqref{symC} extends 
beyond the submatrix $R^{(\lambda)}_{M+1:N, M+1:N}$ as the entire matrix is 
symmetric up to a scaling factor. Hence, Legendre-based convolution matrices 
are exactly banded with bandwidth $M+1$ and this is the \textit{only} 
case where polynomial-based convolution matrices are exactly banded. 

\subsection{Convolution matrices in Jacobi space} \label{SEC:jac}
For Jacobi polynomials, we adopt the most commonly-used normalization, which 
can be found, for example, in \cite[\S 4.2.1]{sze}. In terms of 
hypergeometric function, they are
\begin{equation}
P_n^{(\alpha,\beta)}(x) = \frac{(\alpha+1)_n}{n!} {}_2F_{1}\left( 
\begin{array}{c}-n, n+\alpha+\beta+1 \\ \alpha+1 \end{array}; 
\frac{1-x}{2}\right), ~~~ n \geqslant 0,
\end{equation}
for $\alpha, \beta > -1$.

We first introduce a similar recurrence for the derivatives and the integrals of 
Jacobi polynomials, analogous to \eqref{recT} for Chebyshev polynomials and 
\eqref{recC} for Gegenbauer polynomials, but with one extra term.

\begin{lemma} For any integer $n$, Jacobi polynomials $P^{(\alpha, 
\beta)}_n(x)$ satisfy the following recurrence relation which can be written in 
derivative or integral forms:
\begin{subequations}
\begin{align}
~~~~~~P_{n}^{(\alpha,\beta)}(x) =  
A^{(\alpha,\beta)}_{n+1}\frac{\md}{\md x} 
P_{n+1}^{(\alpha,\beta)}(x) + B^{(\alpha,\beta)}_{n} \frac{\md}{\md x} 
P_{n}^{(\alpha,\beta)}(x) + C^{(\alpha,\beta)}_{n-1} \frac{\md}{\md x} 
P_{n-1}^{(\alpha,\beta)}(x), \label{drecP} \\
\int P_{n}^{(\alpha,\beta)}(x) \md x = A^{(\alpha,\beta)}_{n+1}
P_{n+1}^{(\alpha,\beta)}(x) + B^{(\alpha,\beta)}_{n} P_{n}^{(\alpha,\beta)}(x) 
+ C^{(\alpha,\beta)}_{n-1} P_{n-1}^{(\alpha,\beta)}(x), \label{irecP}
\end{align}
where
\begin{align}
A^{(\alpha,\beta)}_{n+1} &= 
\frac{2(\alpha+\beta+n+1)}{(\alpha+\beta+2n+1)(\alpha+\beta+2n+2)}, \\
B^{(\alpha,\beta)}_{n} &= \frac{2(\alpha-\beta)}{(\alpha+\beta+2n)(\alpha
+\beta+2n+2)}, \label{ABC} \\
C^{(\alpha,\beta)}_{n-1} &= -\frac{2(\alpha+n)(\beta+n)}{(\alpha+\beta+n) 
(\alpha+\beta+2n)(\alpha+\beta+2n+1)}.
\end{align}
\label{recP}
\end{subequations}
Here, we assume $P_{n}^{(\alpha,\beta)}(x) = 0$ for $n<0$.
\end{lemma}
\begin{proof}
See, for example, \cite[Theorem 3.23]{she}.
\end{proof}

Different from \eqref{recT} or \eqref{recC}, \eqref{drecP} and \eqref{irecP} 
both have a middle term on the right-hand side, indexed with $n$. In the 
symmetric case when $\alpha = \beta$, $B_n^{(\alpha, \beta)}$ becomes zero and 
this middle term vanishes. 

Analogous to Theorem \ref{THM:T} and Theorem \ref{THM:C}, a recurrence relation 
for the convolutions of a Jacobi series with Jacobi polynomials can be derived 
using \eqref{drecP}.

\begin{theorem}[Recurrence of convolutions of Jacobi polynomials]
For Jacobi series $f_M(x) = \sum_{m=0}^M a_m P^{(\alpha, \beta)}_m(x)$,
\begin{equation}
\begin{multlined}
\iy f_M(x-t) P^{(\alpha, \beta)}_{n+1}(t) \dt = -\frac{1}{A^{(\alpha, 
\beta)}_{n+1}} \iiy f_M(x-t) P^{(\alpha, \beta)}_n(t) \dd \\
-\frac{B^{(\alpha, \beta)}_{n}}{A^{(\alpha, \beta)}_{n+1}} \iy f_M(x-t) 
P^{(\alpha, \beta)}_n(t) \dt -\frac{C^{(\alpha, \beta)}_{n-1}}{A^{(\alpha, 
\beta)}_{n+1}} \iy f_M(x-t) P^{(\alpha, \beta)}_{n-1}(t) \dt \\
+ \frac{S^{(\alpha, \beta)}_n}{A^{(\alpha, \beta)}_{n+1}}\iy f_M(t) \dt,
\end{multlined}\label{Pn}
\end{equation}
where $x \in [-2, 0]$, $y = x+1 \in [-1, 1]$, and
\begin{equation}
S^{(\alpha, \beta)}_n = 
\frac{2(-1)^{n+1}(\beta)_{n+1}}{(\alpha+\beta+n)(n+1)!}. \label{SP}
\end{equation}
\end{theorem}
Note that the second term on the right-hand side accounts for the middle term in 
\eqref{drecP}. Recognizing the convolutions in \eqref{Pn} as Jacobi series and 
then applying \eqref{irecP} to the first term on the right-hand side, we obtain 
the recurrence relation for the entries of a Jacobi convolution matrix 
$R^{(\alpha, \beta)}$.
\begin{theorem}[Construction of $R^{(\alpha, \beta)}$] 
The entries of the zeroth column of $R^{(\alpha, \beta)}$ are
\begin{subequations}
\begin{equation}
R^{(\alpha, \beta)}_{k,0} =
\begin{cases}
0 & \quad k > M+1, \\[2mm]
\displaystyle A^{(\alpha, \beta)}_k a_{k-1} + B^{(\alpha, \beta)}_k a_{k} + 
C^{(\alpha, \beta)}_k a_{k+1} & \quad 1 \leqslant k \leqslant M+1, \\[2mm]
\displaystyle \sum_{j=1}^{M+1} (-1)^{j+1} \frac{(\beta+1)_j}{j!} R^{(\alpha, 
\beta)}_{j,0} & \quad k = 0,
\end{cases}
\label{PR0}
\end{equation}
with $a_{M+1} = a_{M+2} = 0$. For $n \geqslant 0$,
\begin{equation}
\begin{multlined}
R^{(\alpha, \beta)}_{k, n+1} = \frac{B^{(\alpha, \beta)}_{k}-B^{(\alpha, 
\beta)}_{n}}{A^{(\alpha, \beta)}_{n+1}}R^{(\alpha, \beta)}_{k, n} 
-\frac{C^{(\alpha, \beta)}_{n-1}}{A^{(\alpha, \beta)}_{n+1}}R^{(\alpha, 
\beta)}_{k, n-1} + \frac{A^{(\alpha, \beta)}_{k}}{A^{(\alpha, 
\beta)}_{n+1}}R^{(\alpha, \beta)}_{k-1, n} \\
+ \frac{C^{(\alpha, \beta)}_{k}}{A^{(\alpha, \beta)}_{n+1}}R^{(\alpha, 
\beta)}_{k+1, n} +\frac{S^{(\alpha, \beta)}_{n}}{A^{(\alpha, 
\beta)}_{n+1}}R^{(\alpha, \beta)}_{k, 0},
\end{multlined}\label{PRn}
\end{equation}
\label{PR}
\end{subequations}
where $R^{(\alpha, \beta)}_{:, -1}$ are understood to be zeros.
\end{theorem}

Again, the same stability issue holds us from constructing the Jacobi-based 
convolution matrices by using \eqref{PRn} directly. Fortunately, the symmetry 
persists though the recurrence relation \eqref{PRn} is augmented by the extra 
$R^{(\alpha, \beta)}_{k,n}$ term. Thus, the Jacobi-based convolution matrices 
are almost-banded too. Now, we are in a position to show the symmetry of the 
Jacobi-based convolution matrices, from which the symmetric properties of the 
Chebyshev- and Gegenbauer-based convolution matrices can be easily deduced.
\begin{theorem}[Symmetry of $R^{(\alpha, \beta)}$] \label{THM:symP}
For $M+1 \leqslant k,n \leqslant N$,
\begin{equation}
R_{n, k}^{(\alpha,\beta)} = (-1)^{k+n}\frac{(\alpha+\beta+2n+1)(\alpha +1)_k (\beta +1)_k 
\Big( (\alpha+\beta+1)_n \Big)^2}{(\alpha+\beta+2k+1) (\alpha+1)_n (\beta+1)_n 
\Big( (\alpha+\beta+1)_k \Big)^2}R_{k, n}^{(\alpha,\beta)}. \label{symP}
\end{equation}
\end{theorem}

\begin{proof}
For $k \geqslant M+2$, \eqref{PRn} reduces to 
\begin{equation}
\begin{multlined}
\frac{A^{(\alpha, \beta)}_{n+1}}{B^{(\alpha, \beta)}_k-B^{(\alpha, 
\beta)}_n}R^{(\alpha, \beta)}_{k,n+1} + \frac{C^{(\alpha, 
\beta)}_{n-1}}{B^{(\alpha, \beta)}_k-B^{(\alpha, \beta)}_n} R^{(\alpha, 
\beta)}_{k,n-1} - R^{(\alpha, \beta)}_{k,n} \\
-\frac{A^{(\alpha, \beta)}_k}{B^{(\alpha, \beta)}_k-B^{(\alpha, 
\beta)}_n}R^{(\alpha, \beta)}_{k-1,n} - \frac{C^{(\alpha, 
\beta)}_k}{B^{(\alpha, \beta)}_k-B^{(\alpha, \beta)}_n} R^{(\alpha, 
\beta)}_{k+1,n}=0, 
\end{multlined}\label{PRn1}
\end{equation}
since $R^{(\alpha, \beta)}_{k, 0} = 0$ when $k \geqslant M+2$.

Noting that $R^{(\alpha, \beta)}_{k,n}$ is a rational function of $n$, $k$, and 
$\ua$, we denote the ratio of $R^{(\alpha, \beta)}_{k,n}$ and $R^{(\alpha, 
\beta)}_{n, k}$ by $r(n,k, \ua)$, that is,
\begin{equation}
R^{(\alpha, \beta)}_{k,n} = r(n,k, \ua) R^{(\alpha, \beta)}_{n,k}, \label{ratio}
\end{equation}
with $r(n,n, \ua)=1$.

Substituting \eqref{ratio} into \eqref{PRn1} and dividing all terms by $r(n, 
k, \ua)$, we have
\begin{equation*}
\begin{multlined}
\frac{A^{(\alpha, \beta)}_{n+1}}{B^{(\alpha, \beta)}_k-B^{(\alpha, 
\beta)}_n}\frac{r(n+1,k, \ua)}{r(n,k, \ua)} R^{(\alpha, \beta)}_{n+1,k} + 
\frac{C^{(\alpha, \beta)}_{n-1}}{B^{(\alpha, \beta)}_k - B^{(\alpha, 
\beta)}_n}\frac{r(n-1,k, \ua)}{r(n,k, \ua)} R^{(\alpha, \beta)}_{n-1,k} - 
R^{(\alpha, \beta)}_{n,k} \\ 
-\frac{A^{(\alpha, \beta)}_{k}}{B^{(\alpha, \beta)}_k-B^{(\alpha, 
\beta)}_n}\frac{r(n,k-1, \ua)}{r(n,k, \ua)} R^{(\alpha, \beta)}_{n,k-1} 
-\frac{C^{(\alpha, \beta)}_k}{B^{(\alpha, \beta)}_k-B^{(\alpha, 
\beta)}_n}\frac{r(n,k+1, \ua)}{r(n,k, \ua)} R^{(\alpha, \beta)}_{n,k+1}=0.
\end{multlined}\label{PRn2}
\end{equation*}

Now, swapping $k$ and $n$ in \eqref{PRn1} gives
\begin{equation*}
\begin{multlined}
\frac{A^{(\alpha, \beta)}_{k+1}}{B^{(\alpha, \beta)}_n-B^{(\alpha, 
\beta)}_k}R^{(\alpha, \beta)}_{n,k+1} + \frac{C^{(\alpha, 
\beta)}_{k-1}}{B^{(\alpha, \beta)}_n-B^{(\alpha, \beta)}_k}R^{(\alpha, 
\beta)}_{n,k-1}- R^{(\alpha, \beta)}_{n,k} \\
-\frac{A^{(\alpha, \beta)}_{n}}{B^{(\alpha, \beta)}_n-B^{(\alpha, 
\beta)}_k}R^{(\alpha, \beta)}_{n-1,k} - \frac{C^{(\alpha, 
\beta)}_{n}}{B^{(\alpha, \beta)}_n-B^{(\alpha, \beta)}_k}R^{(\alpha, 
\beta)}_{n+1,k}=0. 
\end{multlined}\label{PRn3}
\end{equation*}
Matching the terms in the last two equations, we obtain two recurrence 
relations for $r(n, k, \ua)$
\begin{equation*}
\frac{r(n,k+1, \ua)}{r(n,k, \ua)} = \frac{A^{(\alpha, \beta)}_{k+1}}{C^{(\alpha, 
\beta)}_k} \quad \mbox{and} \quad \frac{r(n+1,k, \ua)}{r(n,k, \ua)} = 
\frac{C^{(\alpha, \beta)}_n}{A^{(\alpha, 
\beta)}_{n+1}}
\end{equation*}
for any $n, k\geqslant M+1$. Therefore, 
\begin{align*}
& r(n,k, \ua) = r(n,M+1, \ua) \hspace{-3mm} \prod_{j=M+1}^{k-1} 
\frac{A^{(\alpha, \beta)}_{j+1}}{C^{(\alpha, \beta)}_{j}} \\ 
\text{and }\quad &r(n,M+1, \ua) = r(M+1,M+1, \ua) \hspace{-3mm} 
\prod_{j=M+1}^{n-1} \frac{C^{(\alpha, \beta)}_{j}}{A^{(\alpha, \beta)}_{j+1}},
\end{align*}
which, combined, give
\begin{equation}
r(n,k, \ua) = \left(\prod_{j=0}^{n-1} \frac{C^{(\alpha, \beta)}_{j}}{A^{(\alpha, 
\beta)}_{j+1}}\right) \left(\prod_{j=0}^{k-1} \frac{A^{(\alpha, 
\beta)}_{j+1}}{C^{(\alpha, \beta)}_{j}} \right),
\end{equation}
where $r(M+1,M+1,\ua) = 1$ is used. Substituting \eqref{ABC} in gives 
\eqref{symP}.
\end{proof}

Now we are ready to show Theorems \ref{THM:symT} and \ref{THM:symC}. 
\begin{proof}[Proof of Theorem \ref{THM:symT}]
Taking the limit of \eqref{symP} as $\alpha,\beta\to -1/2$, we have 
\begin{equation}\label{symP half}
R_{n,k}^{(-\frac{1}{2},-\frac{1}{2})} = (-1)^{n+k}\frac{k}{n}\displaystyle 
\frac{\left(\left( \frac{1}{2} \right)_k\right)^2 
\left(n!\right)^2}{\left(\left(\frac{1}{2} \right)_n\right)^2\left(k!\right)^2} 
\ R_{k,n}^{(-\frac{1}{2},-\frac{1}{2})}.
\end{equation}
Relating Chebyshev polynomial $T_n(x)$ to Jacobi polynomial 
$P^{(-1/2,-1/2)}_n(x)$ by
\begin{equation}
P_n^{(-\frac{1}{2},-\frac{1}{2})}(x) = \frac{\left(\frac{1}{2}\right)_n}{n!} 
T_n(x),
\end{equation}
we obtain
\begin{equation}
R_{k,n}^{(-\frac{1}{2},-\frac{1}{2})} = \frac{\left(\frac{1}{2}\right)_n 
k!}{\left(\frac{1}{2}\right)_k n!} R_{k,n} \label{scalingT}
\end{equation}
for any $k,n \geqslant 0$. Finally, substituting \eqref{scalingT} into 
\eqref{symP half} gives \eqref{symT}.
\end{proof}

\begin{proof}[Proof of Theorem \ref{THM:symC}]
The limit of \eqref{symP} as $\alpha,\beta\to \lambda-1/2$ is
\begin{equation}\label{symP G}
R_{n,k}^{(\lambda-\frac{1}{2},\lambda-\frac{1}{2})} = (-1)^{k+n} 
\frac{\lambda+n}{\lambda+k} \frac{ \left((2 \lambda)_n\right)^2 
\left(\left(\lambda+\frac{1}{2}\right)_k \right)^2 } {\left((2 
\lambda)_k\right)^2   
\left(\left(\lambda+\frac{1}{2}\right)_n\right)^2}R_{k,n}^{(\lambda-\frac{1}{2
},\lambda-\frac{1}{2})} 
\end{equation}
With the scaling between $C_n^{(\lambda)}$ and 
$P_{n}^{(\lambda-1/2,\lambda-1/2)}$ 
\begin{equation*}
P_{n}^{(\lambda-\frac{1}{2},\lambda-\frac{1}{2})} = \frac{\left(\lambda+\frac{1}{2}\right)_n}{\left(2\lambda\right)_n} C_n^{(\lambda)}(x), 
\end{equation*}
we have 
\begin{equation}\label{scalingC}
R_{k,n}^{(\lambda-\frac{1}{2},\lambda-\frac{1}{2})}  
=\frac{ \left(\lambda +\frac{1}{2}\right)_k (2 \lambda)_n}{(2\lambda)_k 
\left(\lambda+\frac{1}{2}\right)_n} R_{k,n}^{(\lambda)}
\end{equation}
for all $k,n\geqslant 0$. Combining \eqref{scalingC} and \eqref{symP G} yields
\eqref{symC}.
\end{proof}

The symmetric relation \eqref{symP} cannot be used directly due to the 
arithmetic overflow for large $k$ and $n$. Instead, we cancel the common factors 
in the numerator and the denominator and match up the factors of similar 
magnitude to obtain an equivalent, but numerically more manageable formula by 
noting that \eqref{symP} is only needed for $k>n$:
\begin{equation}
R_{n, k}^{(\alpha, \beta)} = (-1)^{k+n}\frac{\alpha+\beta+2n+1}{\alpha+\beta+2k+1}
\prod_{j=n}^{k-1}\frac{j+\alpha+1}{j+\alpha+\beta+1}
\prod_{j=n}^{k-1}\frac{j+\beta+1}{j+\alpha+\beta+1} R_{k, n}^{(\alpha, \beta)}. 
\label{symPs}
\end{equation}

When we recurse for the top rows, we need to rewrite \eqref{PRn} as we do for 
the Chebyshev and Gegenbauer cases:
\begin{equation}
\begin{multlined}
R^{(\alpha, \beta)}_{k-1, n} = \frac{B^{(\alpha, \beta)}_{n}-B^{(\alpha, 
\beta)}_{k}}{A^{(\alpha, \beta)}_{k}}R^{(\alpha, \beta)}_{k, n} 
+\frac{C^{(\alpha, \beta)}_{n-1}}{A^{(\alpha, \beta)}_{k}}R^{(\alpha, 
\beta)}_{k, n-1} + \frac{A^{(\alpha, \beta)}_{n+1}}{A^{(\alpha, 
\beta)}_{k}}R^{(\alpha, \beta)}_{k, n+1} \\
- \frac{C^{(\alpha, \beta)}_{k}}{A^{(\alpha, \beta)}_{k}}R^{(\alpha, 
\beta)}_{k+1, n} - \frac{S^{(\alpha, \beta)}_{n}}{A^{(\alpha, 
\beta)}_{k}}R^{(\alpha, \beta)}_{k, 0}.
\end{multlined}\label{PRnC}
\end{equation}

Now we recap the algorithms for constructing the Gegenbauer- and the 
Jacobi-based convolution matrices simultaneously:
\begin{algorithm}[H]
\caption{Construction of the convolution matrix $R^{(\lambda)}/R^{(\alpha, 
\beta)}$}\label{ALGO:geja}
 %\begin{multicols}{2}
\begin{algorithmic}[1]
\STATE Construct the non-zero entries in the zeroth column 
$R^{(\lambda)}_{:, 0}/R^{(\alpha, \beta)}_{:, 0}$ using \eqref{CR0} 
/\eqref{PR0}.
\STATE Calculate the non-zero entries on and below the main diagonal using 
\eqref{CRn} /\eqref{PRn}.
\STATE Calculate the non-zero entries above the main diagonal in rows $M+1$ to 
$N-1$ using \eqref{symC}/\eqref{symPs}.
\STATE Calculate the entries above the main diagonal in the top $M+1$ rows 
using \eqref{CRnC}/\eqref{PRnC}. 
\end{algorithmic}
\end{algorithm}

The complexity of Algorithm \ref{ALGO:geja} is $\bigO(MN)$, same as that of 
Algorithm \ref{ALGO:cheb}. To keep the computational cost minimal in practice, 
particularly for the Jacobi-based convolution matrices, we precompute and store 
$\{A^{(\alpha, \beta)}_j\}_{j=1}^{M+N+1}$, $\{B^{(\alpha, 
\beta)}_j\}_{j=0}^{M+N+1}$, $\{C^{(\alpha, \beta)}_j\}_{j=0}^{M+N+1}$, and 
$\{S^{(\alpha, \beta)}_j\}_{j=1}^{M+N}$ for the use in steps 2 and 4. For step 
3, we can first compute the ratio factor on the right-hand side of 
\eqref{symPs} for $(M+1)$-th column and the ratio factor for a subsequent column 
can be updated from that of the last column accordingly.

\textbf{Example 5:} We repeat Example 3 for Gegenbauer- and Jacobi-based 
convolution matrices. Again, the coefficient vector $\ua$ is randomly generated
with $|a_m|\leqslant 1$. Figure \ref{FIG:geja} shows the entrywise absolute 
error in \subref{FIG:gegen} the Gegenbauer convolution matrix for $\lambda=2$ 
and \subref{FIG:jac} the Jacobi convolution matrix for $\alpha=2$ and 
$\beta=3/2$, where $M=1000$ and $N=5000$ for both experiments. The largest 
entrywise error in Figure \ref{FIG:gegen} is $1.14\times 10^{-11}$, whereas 
$2.40 \times 10^{-11}$ in Figure \ref{FIG:jac}. In these experiments, the 
magnitudes of the entries range \subref{FIG:gegen} from $\bigO(10^{-324})$ to 
$\bigO(10^{3})$ and \subref{FIG:jac} from $\bigO(10^{-324})$ to $\bigO(1)$, 
respectively.

\begin{figure}[t]
\centering
\begin{subfigure}[b]{0.45\textwidth}
\includegraphics[scale=0.5]{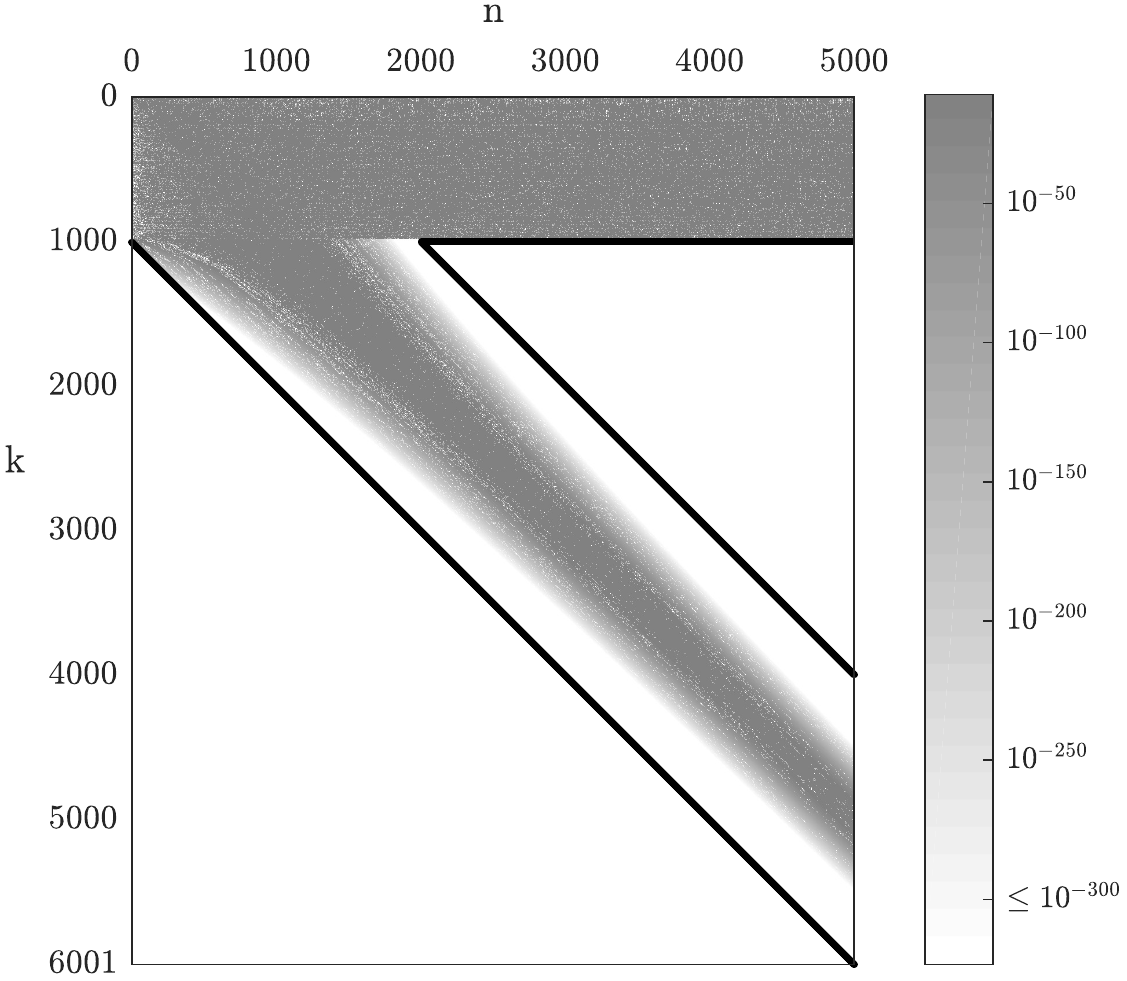}\caption{Gegenbauer 
($\lambda = 2$)} \label{FIG:gegen}
\end{subfigure}
\begin{subfigure}[b]{0.45\textwidth}
\includegraphics[scale=0.5]{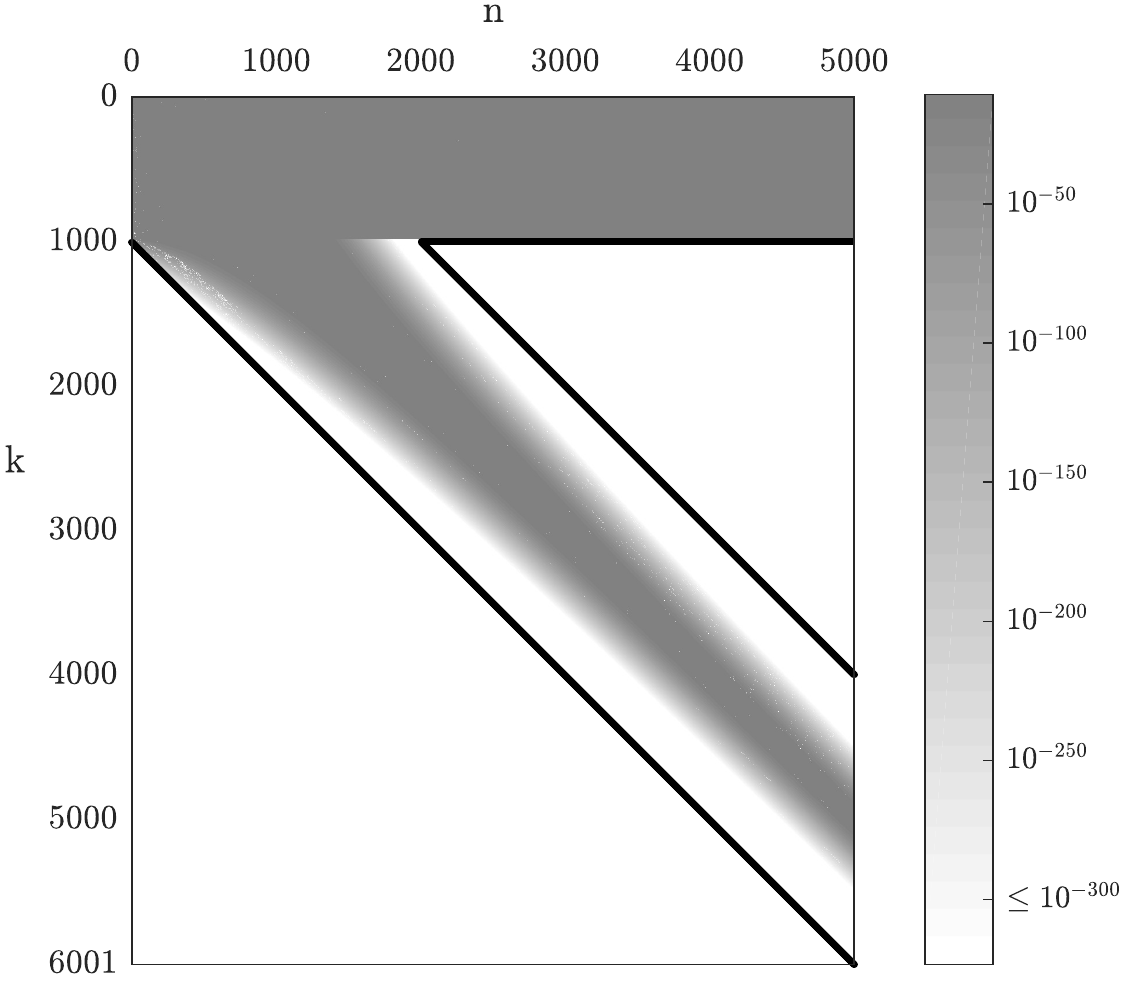}\caption{Jacobi 
($\alpha=2, \beta=3/2$)}\label{FIG:jac}
\end{subfigure}
\caption{Entrywise error of (a) the Gegenbauer-based convolution matrix and (b) 
the Jacobi-based convolution matrix obtained using Algorithm \ref{ALGO:geja}. 
For both examples, $M=1000$ and $N=5000$.}\label{FIG:geja}
\end{figure}

\section{Laguerre-based convolution matrices} \label{SEC:lag}
If a smooth function defined on a semi-infinite domain decays to zero fast 
enough, it can be approximated by a series of weighted Laguerre polynomials
\begin{equation}
L_n^W(x) = e^{-x/2}L_n(x),
\end{equation}
where $L_n(x)$ is the Laguerre polynomial of degree $n$. Thus we consider the 
approximation of the convolution operator defined by such a function using 
Laguerre-based convolution matrices. We start with the following lemma 
which can be found in many standard texts, for example, \cite[(18.17.2)]{dlmf}.
\begin{lemma}[Convolution of Laguerre polynomials]\label{LEM:convL}
\begin{equation}
\int_{0}^{x}L_m(x-t)L_{n}(t)\dt = L_{m+n}(x)-L_{m+n+1}(x),
\end{equation}
for $x \in [0, \infty]$.
\end{lemma}

Consider the convolution of continuous decaying functions $f(x)$ and $g(x)$ 
defined on $[0, \infty]$:
\begin{equation}
V[f](g) = h(x) = \int_0^x f(x-t) g(t)\dt, ~~~~~ x \in [0, \infty].
\label{VV}
\end{equation}
Suppose that $f(x)$ and $g(x)$ are represented by infinite weighted Laguerre 
series
\begin{equation}
f(x) = e^{-x/2}\sum_{m=0}^{\infty} a_m L_m(x) \quad \mbox{and} \quad 
g(x) = e^{-x/2}\sum_{n=0}^{\infty} b_n L_n(x), \qquad x \in [0, 
\infty],
\end{equation}
and assume that the convolution
\begin{equation}
h(x) = e^{-x/2} \sum_{k=0}^{\infty} c_k L_k(x), \qquad x \in [0, \infty],
\end{equation}
so that $\underline{c} = R^L\underline{b}$, where $R^L$ is the Laguerre 
convolution matrices generated by $\ua$. By Lemma \ref{LEM:convL}, the entries 
of $R^L$ are 
explicitly known.  
\begin{theorem}[Construction of $R^L$]
The matrix approximation of convolution operator $V[f]$ in Laguerre space is 
the difference of two lower triangular Toeplitz matrices, where the second one 
is obtained by adding one row of zeros on the top of the first:
\begin{equation}
R^L=
\begin{pmatrix}
  a_0 & 0 & 0 & 0 & \cdots \\[-0.4em]
  a_1 & a_0 & 0 & 0 & \ddots \\[-0.4em]
  a_2 & a_1 & a_0 & 0 & \ddots \\[-0.4em]
  a_3 & a_2 & a_1 & a_0 & \ddots \\[-0.4em]
  \vdots & \ddots & \ddots & \ddots & \ddots
\end{pmatrix}
-
\begin{pmatrix}
  0 & 0 & 0 & 0 & \cdots \\[-0.4em]
  a_0 & 0 & 0 & 0 & \ddots \\[-0.4em]
  a_1 & a_0 & 0 & 0 & \ddots \\[-0.4em]
  a_2 & a_1 & a_0 & 0 & \ddots \\[-0.4em]
  \vdots & \ddots & \ddots & \ddots & \ddots
\end{pmatrix}
.
\label{L}
\end{equation}
\end{theorem}

When $f(x)$ and $g(x)$ are approximated by finite weighted Laguerre series, 
the convolution matrix $R^L$ becomes a banded lower-triangular Toeplitz matrix, 
as shown in Figure \ref{FIG:lag_structure} and the Toeplitz structure allows 
the fast application of $R^L$ to $\ub$ with the aid of FFT.

\textbf{Example 6:}  We consider the convolution of
\begin{equation*}
f(x) = \frac{1}{2} x^2 e^{-x} ~\mbox{ and }~ g(x) = - \frac{1}{3} 
\left(\cos\frac{\sqrt{3}}{2}x + \sqrt{3} 
\sin\frac{\sqrt{3}}{2}x\right)e^{-3x/2}, \quad x\in [0, \infty],
\end{equation*}
which are the functions from Example 4, only except that the constant $1/3$ is 
removed from the original $g(x)$ so that both the functions decay to zero at 
infinity. These two functions can be approximated by weighted Laguerre series of 
degree $2$ and $54$, respectively. In Figure \ref{FIG:lag_example}, the 
pointwise error in the computed approximant against the exact convolution 
\begin{equation*}
h(x) = (f\ast g)(x) = -\frac{1}{3} e^{-3 x/2} \left[e^{x/2} 
\left(x^2-x-1\right) 
+\sqrt{3} \sin \left(\frac{\sqrt{3} x}{2}\right)+\cos \left(\frac{\sqrt{3} 
x}{2}\right)\right]
\end{equation*}
is shown up to $x=10^4$ and the maximum error throughout $[0, \infty]$ is 
approximately $4.4 \times 10^{-15}$, occurring at about $x=2.9$.
\begin{figure}[h!]
\centering
\hspace{4mm}
\begin{subfigure}[b]{0.35\textwidth}
\includegraphics[scale=0.2]{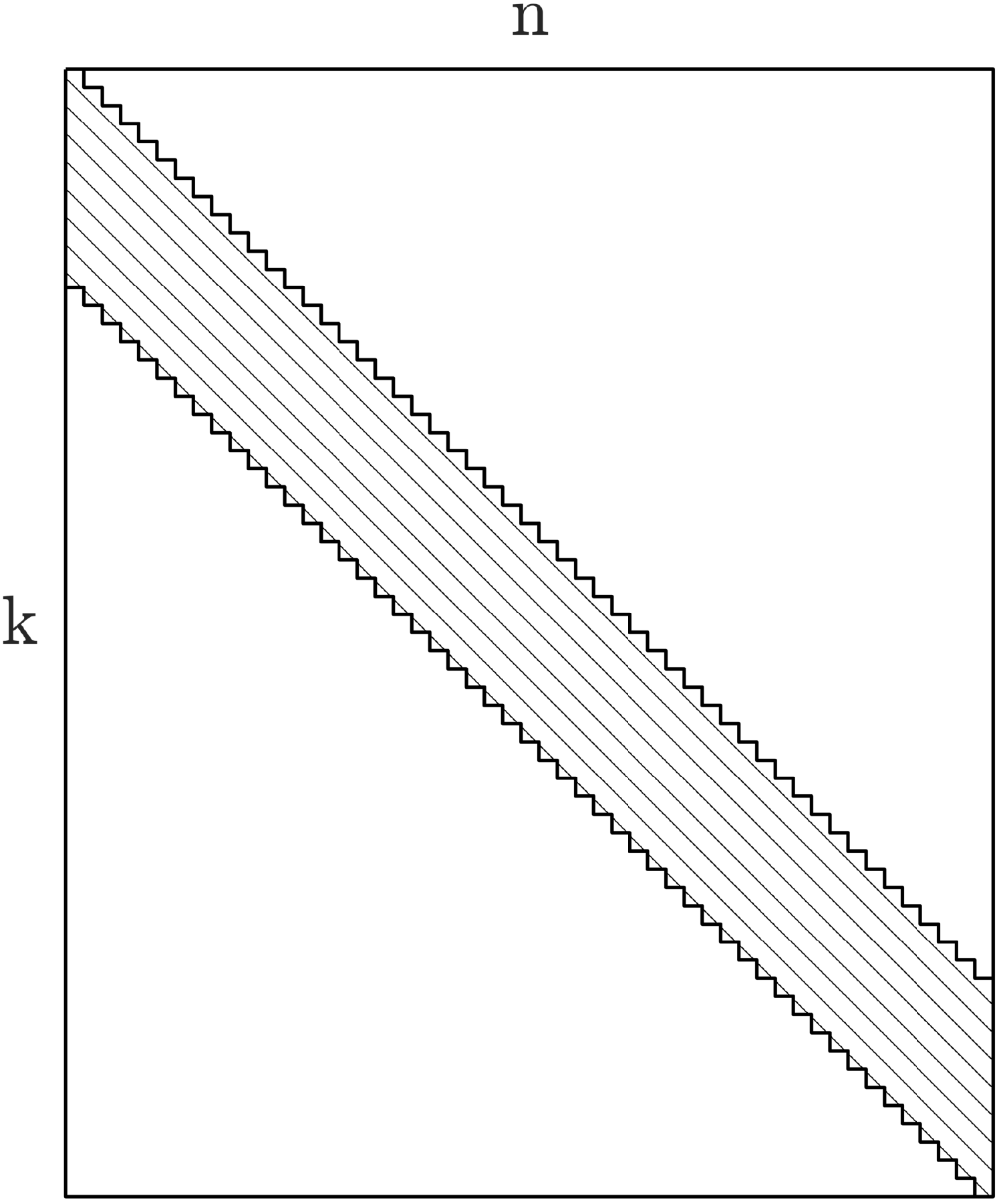}
\caption{~~~~~~~~} 
\label{FIG:lag_structure}
\end{subfigure}
\hspace{-5mm}
\begin{subfigure}[b]{0.63\textwidth}
\includegraphics[scale=0.24]{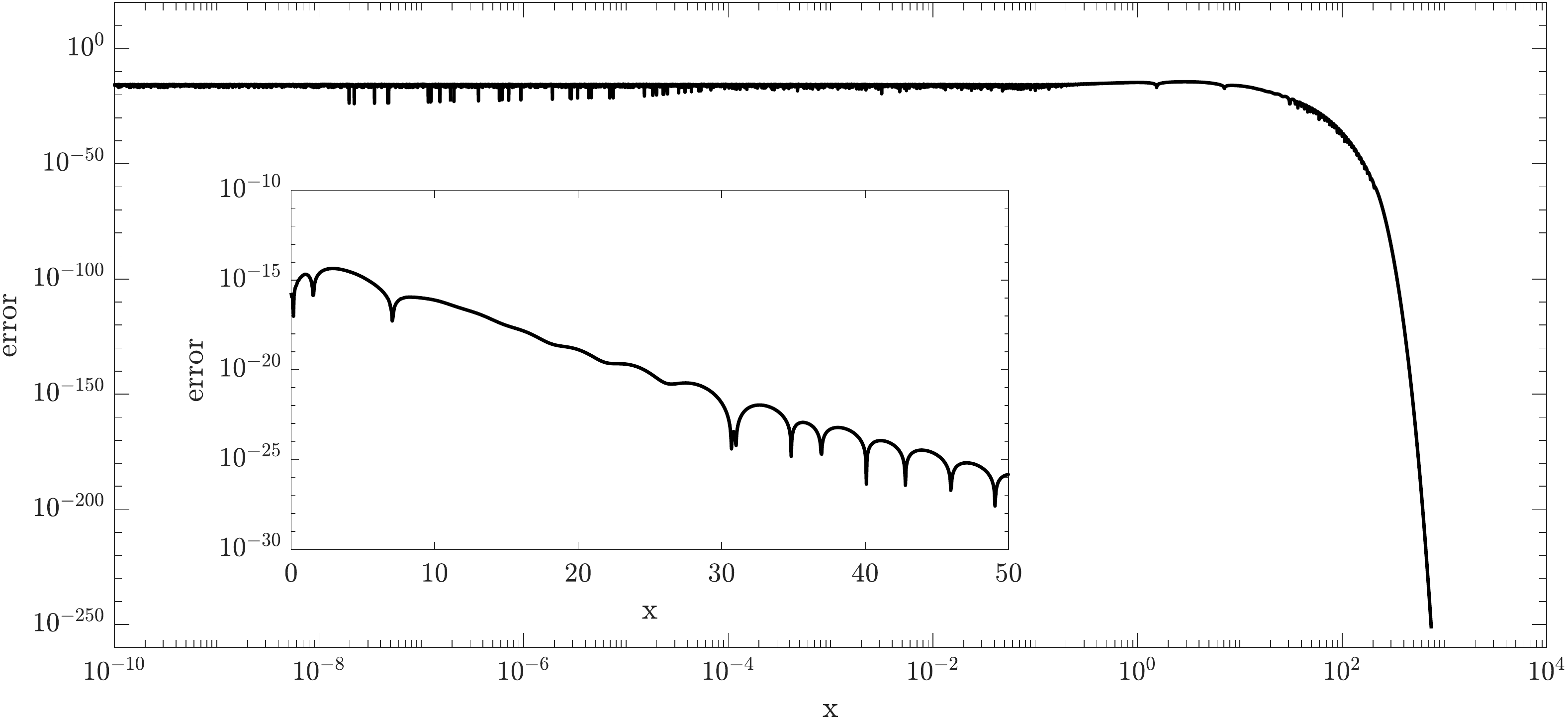}
\caption{ }
\label{FIG:lag_example}
\end{subfigure}
\caption{(a) A schematic of Laguerre convolution matrices, where we use stripes 
to indicate the Toeplitz structure. Outside the striped region, all entries are 
exactly zero. (b) Absolute error in the computed approximation to the 
convolution, where the inset is a close-up for $[0, 50]$.} \label{FIG:lag}
\end{figure}

\section{Closing remarks} \label{SEC:conc}
While we have focused exclusively on the left-sided convolution operator, with 
a few minor changes the framework we have presented can be extended to the 
right-sided convolution operator
\begin{equation}
V[f](g) = \int_{x-1}^1 f(x-t) g(t)\md t,~~~~~x \in [0, 2], 
\label{VVV}
\end{equation}
with $f(x)$ and $g(x)$ compactly supported on $[-1, 1]$ and the right-sided 
convolution matrices also enjoy similar recurrences and symmetric properties.

In Example 4, we have shown how convolution matrices can be employed to solve 
convolution integral equations. With the almost-banded structure of the 
convolution matrices, a fast spectral method can be developed based on the 
framework of infinite-dimensional linear algebra \cite{olv2} for solving 
convolution integro-differential equations of Volterra type:
\begin{equation}
\sum_{j=0}^{J} b^j(x) \frac{\md^j u(x)}{\md x^j}+ b^{J+1}(x)\int_a^x K(x-t)u(t) 
\dt = s(x),
\end{equation} 
where $b^j(x)$ for $0 \leqslant j \leqslant J+1$ are smooth functions on $[-1, 
1]$ and the convolution kernel $K(x-t)$ is smooth or weakly singular. We will 
report this line of research in a future work.

\section*{Acknowledgments}

We would like to thank Anthony P. Austin (Argonne), Alex Townsend (Cornell), 
and Haiyong Wang (HUST) for their extremely valuable commentary on an early 
draft of this paper, Marcus Webb (KU Leuven) for very helpful discussions, and 
Chuang Sun (MathWorks) for sharing his \textsc{Mathematica} tricks with the 
first author. Finally, we thank the anonymous referees, whose careful reading 
and feedback led us to improve our work.

\bibliographystyle{siam}
\bibliography{conv}

\end{document}